\documentclass[12pt]{amsart}
\usepackage{color,graphicx,array, amssymb, amscd}
\newtheorem{Theorem}{Theorem}[section]
\newtheorem{Lemma}[Theorem]{Lemma}
\newtheorem{Proposition}[Theorem]{Proposition}
\newtheorem{Corollary}[Theorem]{Corollary}
\theoremstyle{definition}
\newtheorem{Definition}{Definition}[section]
\theoremstyle{remark}

\newtheorem{Remark}[Theorem]{Remark} 

\numberwithin{equation}{section}
\def\Vec#1{\mbox{\boldmath $#1$}}
\newcommand{\C}{\mathbb C}
\newcommand{\D}{\mathbb D}

\newcommand{\F}{\mathcal F}
\newcommand{\A}{\mathcal A}
\newcommand{\K}{\mathcal K}
\newcommand{\V}{\mathcal V}
\newcommand{\R}{\mathcal R}
 \newcommand{\W}{\mathcal W}
\newcommand{\G}{\mathcal G}
\newcommand{\B}{\mathcal B}
\newcommand{\mH}{\mathcal H}
\newcommand{\N}{\mathcal N}

\newcommand{\ad}{\operatorname{Ad}}

\newcommand{\id}{\operatorname{id}}

\renewcommand{\Re}{\operatorname {Re}}

\newcommand{\nm}{\nabla^{\mu}}
\newcommand{\nt}{{}^{(t)} \nabla}
\newcommand{\can}{{}^{(-1)}\nabla}
\newcommand{\anti}{{}^{(1)}\nabla}
\newcommand{\neutral}{{}^{(0)}\nabla}
\newcommand{\LG}{\Lambda G}
\newcommand{\LGC}{\Lambda G^{\mathbb C}}
\newcommand{\LBC}{\Lambda B^{\mathbb C}}
\newcommand{\LHC}{\Lambda H^{\mathbb C}}

\newcommand{\NLGC}{\Lambda^{-} G^{\mathbb C}}
\newcommand{\PNLGC}{\Lambda^{\pm} G^{\mathbb C}}
\newcommand{\LGCD}{\Lambda \G^{\mathbb C}}
\newcommand{\LBCD}{\Lambda \B^{\mathbb C}}
\newcommand{\LHCD}{\Lambda \mH^{\mathbb C}}
\newcommand{\PLGCD}{\Lambda^{+} \G^{\mathbb C}}
\newcommand{\NLGCD}{\Lambda^{-} \G^{\mathbb C}}

\newcommand{\PNLGCD}{\Lambda^{\pm} \G^{\mathbb C}}
\newcommand{\NLGCDN}{\Lambda^{-}_{*} \G^{\mathbb C}}
\newcommand{\LGD}{\Lambda \G}
\newcommand{\sk}{\operatorname {skew}}
\newcommand{\sym}{\operatorname {sym}}
\newcommand{\trace}{\operatorname {tr}_g}
\newcommand{\be}{\begin{equation*}}
\newcommand{\ee}{\end{equation*}}
\renewcommand{\P}{\mathcal P}
\renewcommand{\l}{\lambda}
\renewcommand{\H}{\mathcal H}
\begin{document}
\title{A loop group method for affine harmonic maps into Lie groups}
 \author[J. F.~Dorfmeister]{Josef F. Dorfmeister}
 \address{Fakult\"at f\"ur Mathematik, 
 TU-M\"unchen, 
 Boltzmann str. 3,
 D-85747, 
 Garching, 
 Germany}
 \email{dorfm@ma.tum.de}
\author[J.~Inoguchi]{Jun-ichi Inoguchi}
 \address{Department of Mathematical Sciences, 
 Faculty of Science,
 Yamagata University, 
 Yamagata, 990--8560, Japan}
 \email{inoguchi@sci.kj.yamagata-u.ac.jp}
 \thanks{The second named author is partially supported by Kakenhi 
24540063}
 \author[S.-P.~Kobayashi]{Shimpei Kobayashi}
 \address{Department of Mathematics, Hokkaido University, 
 Sapporo, 060-0810, Japan}
 \email{shimpei@math.sci.hokudai.ac.jp}
 \thanks{The third named author is partially supported by Kakenhi 
23740042, 26400059}
\subjclass[2010]{Primary~58E20~53C43, Secondary~22E65~22E25.}
\keywords{Lie group; harmonic map; generalized Weierstrass type representation; solvable Lie groups}
\date{\today}
\pagestyle{plain}

\begin{abstract} 
 We generalize the Uhlenbeck-Segal theory for harmonic maps 
 into compact semi-simple Lie groups to general 
 Lie groups equipped with
 torsion free bi-invariant connection.
\end{abstract}
\maketitle
\section*{Introduction}
 Harmonic maps of Riemann surfaces into compact semi-simple Lie groups,
 equipped with a
 bi-invariant Riemannian metric, have been paid much attention to by
 differential geometers
 as well as by mathematical physicists. In fact, harmonic maps of Riemann 
 surfaces into 
 compact Lie groups equipped with a bi-invariant Riemannian metric 
 are called \textit{principal chiral models} and intensively studied as 
 toy models of gauge theory in mathematical physics \cite{Zak}.

 Uhlenbeck established a fundamental theory of harmonic maps into
 the unitary group $\mathrm{U}_n$, \cite{Uhlenbeck}. In particular she proved a factorization theorem for
 harmonic $2$-spheres, the so called \textit{uniton factorization}.
 Segal showed that Uhlenbeck's theory actually works for any compact Lie
 group \cite{Segal}.
 Pedit, Wu and the first named author of the present paper
 generalized the Uhlenbeck-Segal theory to harmonic maps into
 compact Riemannian symmetric spaces, 
 now referred as to the \textit{generalized Weierstrass type representation} 
 \cite{DPW}.

 It turned out that the compactness of the target space is not
 necessary, as long as one only considers surfaces away from singularities and considers
 groups with bi-invariant metric, Riemannian or pseudo-Riemannian. 
 From a global point of view the construction principle will generally
 produce surfaces with singularities.
 A typical example for this are the spacelike CMC surfaces in 
 Minkowski 3-space. 
 In this case one considers harmonic maps into the 
 Riemannian symmetric space
 $\mathbb{H}^2=\mathrm{SL}_{2}\mathbb{R}/\mathrm{SO}_2$. 
 Harmonic maps into $\mathbb{H}^2$ are closely related to the so-called  
 $\mathrm{tt}^*$-geometry \cite{DGR}. 
 It is a very important and difficult problem to find (and describe) globally smooth
 solutions for non-compact target spaces.

 If one wants to generalize the loop group approach for 
 harmonic maps into symmetric spaces to general homogeneous 
 spaces, where the Lie group 
 only has a left-invariant metric, one encounters a 
 completely new situation.
 Clearly, the case of harmonic maps into Lie groups is a 
 first interesting case. Since abelian groups always have 
 bi-invariant metrics, the next interesting case is the one 
 of $3$-dimensional Lie groups (with left-invariant metric).

 As a matter of fact, and 
 largely independent of the issues discussed above,
 during the last ten years or so, minimal surfaces into $3$-dimensional Lie
 groups have been studied extensively. In particular, minimal surfaces in
 3-dimensional Lie groups, equipped with a left-invariant metric,
 have been investigated intensively.
 With the exception of the space $\mathbb{S}^2\times\mathbb{R}$,  the other seven 
 model spaces of Thurston geometries obviously have or can be given the
 structure of a Lie group. These are the following spaces; 
 the Euclidean $3$-space $\mathbb{E}^3$, 
 the unit $3$-sphere $\mathbb{S}^3$, 
 the hyperbolic $3$-space $\mathbb{H}^3$, 
 the model space $\mathrm{Nil}_3$
 of nilgeometry, the universal covering group
 $\widetilde{\mathrm{SL}_2\mathbb{R}}$,
 the space $\mathrm{Sol}_3$ of solvgeometry and the product space
 $\mathbb{H}^2\times \mathbb{R}$.
 The metrics on these groups  are generally only left-invariant
 with respect to the Lie group structure. Only
 the  Euclidean $3$-space $\mathbb{E}^3$ and the 
 $3$-sphere $\mathbb{S}^3$ admit bi-invariant Riemannian metrics.

 Since the generalized Weierstrass type representation for
 harmonic maps usually requires a {\it bi-invariant metric} on 
 some related Lie group,
 we can not expect this scheme to work unchanged 
for harmonic maps into Lie
 groups,
 if these Lie groups only carry a left-invariant metric.
 In fact, harmonic maps into general Lie groups do not even admit a
 zero-curvature
 representation in general.
 For instance, 
 all the explicit examples  in \cite{DIK2} of minimal 
 surfaces (except vertical planes) in $\mathrm{Nil}_3$ do not satisfy
 the zero-curvature representation described 
 in Proposition \ref{ZCRforG}.
 On the other hand, it should be remarked that the
 harmonicity equation makes sense for maps from
 Riemann surfaces into Lie groups equipped with any {\it affine connection}.

 Higaki \cite{Higaki} pointed out that if 
 a map $\varphi:M\to G/H$ of a Riemann surface into a 
 reductive homogeneous space $G/H$ 
 is harmonic with respect to the \textit{canonical connection} of $G/H$
 and if its torsion vanishes along $\varphi$, then
 $\varphi$ is \textit{equiharmonic}, that is, harmonic with
 respect to any $G$-invariant metric.
 This result indicates that there may exist
 a large class of harmonic maps with particularly nice properties.
 It is also natural to expect that harmonic maps in such a nice class
 may even admit an explicit construction scheme.
 From our point of view,
 harmonic maps in such a class should admit a loop group method.

 In this paper we develop a loop group theory for harmonic maps into 
 Lie groups which are equipped with a bi-invariant affine connection, 
 namely \textit{affine harmonic maps}.
 Instead of considering the Levi-Civita connection on Lie groups, 
 we will use a natural, torsion free, bi-invariant connection, called
 \textit{neutral connection},  to study maps from
 Riemann surfaces to Lie groups.
 We shall show that smooth maps into a general Lie group $G$ 
 equipped with the neutral connection admit a loop group formulation.
 Based on this fundamental result, we generalize the generalized 
 Weierstrass type representation to
 harmonic maps into Lie groups equipped with the neutral connection.
 We note that a Lie group $G$, equipped with the neutral connection, 
 is looked at as the affine symmetric space $G\times{G}/G$.
 If $G$ is semi-simple, any neutral connection
 coincides with the Levi-Civita connection of the
 bi-invariant metric induced by the Killing form.
 In this sense, the present paper is a generalization
 of the loop group method originally developed by Uhlenbeck and Segal.

 We would like to emphasize that we establish a
 loop group method for affine harmonic maps into
 \textit{any} Lie group, in particular also into
 Lie groups which do
 not have any bi-invariant metric.
 As a consequence, in a sense,  the particularly new feature
 of this paper is a treatment of harmonic maps into solvable Lie groups
 (and, as a special case, into  nilpotent Lie groups), since, generally, these groups  do not admit
 any bi-invariant metric.

 In this context we would like to mention that 
 harmonic maps into Lie groups, especially nilpotent or solvable 
 Lie groups, equipped with left-invariant affine 
 connections have applications to probability theory, since it is known 
 that harmonic maps have a probabilistic characterization:
 A smooth map between Riemannian manifolds is a harmonic map if and only if 
 it sends Brownian motions to martingales \cite{Kendall, Meyer}.  
 Martingales and harmonic maps into Lie groups equipped 
 with left-invariant affine connections have been studied in \cite{Arnaudon,Stelmastchuk}.

 This paper is organized as follows:
 In Section \ref{sc:Pre}, we will briefly give preliminary 
 results on vector bundle valued differential forms. 
 In Section  \ref{sc:affineharmonic}, 
 affine harmonic maps from a Riemannian manifold into an affine manifold will be 
 discussed. 
 In Section \ref{sc:affineLiegroup}, affine harmonic maps into any Lie group $G$ 
 will be considered. 
 It is known that 
 all left-invariant affine connections on a Lie group  $G$
 are given by bilinear maps $\mu$ on the Lie algebra of $G$, and are therefore denoted by  $\nm$.
 Then $\nm$-harmonic maps
 from Riemann surfaces 
 into $G$, where $G$ carries a left-invariant affine connection $\nm$ defined
 by a  skew-symmetric map $\mu$,  will be characterized by a loop of flat connections in Theorem 
 \ref{thm:familyofconnections}.
 In particular, we consider Lie groups with 
 left-invariant metrics
 and harmonic maps from Riemann surfaces into these Lie groups with the Levi-Civita connection. 
 We will give some simple and fundamental examples of 
 neutral harmonic maps in Section \ref{sc:BasicExamples}.
 In Section \ref{sc:DPW}, the generalized Weierstrass type representation for 
 neutral harmonic maps into Lie groups will be presented. 
 In this section, the Lie group  
 is a connected real analytic Lie group 
 admitting a faithful finite dimensional representation and 
 its linear complexification is simply-connected.
 In the final Section \ref{sc:Examples}, 
 as an example of the generalized Weierstrass type representation, we will discuss 
 the case of $3$-dimensional solvable Lie groups in detail. In Theorem 
 \ref{thm:repforsol}, 
 we give a representation formula for neutral harmonic maps into 
 $3$-dimensional solvable Lie groups. 
 The class of solvable Lie groups considered in Section \ref{sc:Examples} 
 includes the following model spaces of Thurston geometry:
 $\mathbb E^3$, $\mathbb H^3$, $\mathbb H^2 \times \mathbb R$ 
 and $\mathrm{Sol}_3$. In particular we introduce a new class 
 of surfaces in $\mathbb H^3$ which admit holomorphic 
 constructions.

\section{Preliminaries}\label{sc:Pre}
\subsection{Basic facts}
 Let $M$ be a manifold and $E$ a vector bundle over $M$ and 
 denote by $\varGamma(E)$ 
 the space of all smooth sections of the vector bundle $E$.
 The space $\varGamma(\wedge^{r}T^{*}M\otimes{E})$ is 
 denoted by $\Omega^{r}(E)$. An element of $\Omega^{r}(E)$ is called an
 $E$-\textit{valued} $r$-\textit{form} on $M$. 

 In case $E=M\times{V}$ is a trivial  vector bundle over $M$ with 
 standard fiber $V$, then $\Omega^{r}(M\times{V})$ is denoted by 
 $\Omega^{r}(M; V)$.
 An element of $\Omega^{r}(M; V)$ is called a $V$-\textit{valued} 
 $r$-\textit{form} on $M$.
 By definition, for $\alpha\in \Omega^{r}(M; V)$ 
 and $X_{1},X_{2},\cdots, X_{r}\in \varGamma(TM)$, 
 $\alpha(X_1,X_2,\cdots,X_r)\in C^{\infty}(M,V)$.

 Next let $G$ be a Lie group with Lie algebra $\mathfrak{g}$.
 Take a bilinear map $\mu:\mathfrak{g}\times \mathfrak{g}\to
 \mathfrak{g}$. 
 Then for $\alpha$, $\beta\in \Omega^{1}(M;\mathfrak{g})$, we define a
 $\mathfrak{g}$-valued $2$-form $\mu(\alpha\wedge \beta)$ by 
 \be
 \mu(\alpha\wedge \beta)(X,Y):=
 \mu(\alpha(X),\beta(Y))-\mu(\alpha(Y),\beta(X))
 \ee
 for any sections $X, Y  \in \varGamma(TM)$.
 Moreover the \textit{symmetric part} $\sym \mu$ and 
 the \textit{skew-symmetric part} $\sk \mu$ of $\mu$ are defined by
\begin{equation*}
(\sym \mu) (X,Y) := \tfrac{1}{2}\mu (X,Y) + \tfrac{1}{2}
  \mu (Y,X), \;\;
 (\sk \mu )(X,Y):= \tfrac{1}{2}\mu (X,Y)
   - \tfrac{1}{2} \mu(Y,X).
\end{equation*}
 for any sections $X, Y  \in \varGamma(TM)$.
 It is easy to check that the following relations hold 
 for any $\alpha$, $\beta \in \Omega^{1}(M;\mathfrak{g})$:
\begin{equation}\label{eq:sym-skew}
  (\sym \mu) (\beta \wedge \alpha) = -   (\sym \mu) (\alpha \wedge \beta), \;\;
  (\sk \mu) (\beta \wedge \alpha) =  (\sk \mu) (\alpha \wedge \beta).
\end{equation}
 Let us denote by $\theta$ the left-invariant Maurer-Cartan form on 
 a Lie group $G$. 
 By definition, $\theta$ is a $\mathfrak{g}$-valued $1$-form on $G$. 
 The $\mathfrak{g}$-valued $2$-form $[\theta \wedge \theta]$ 
 is computed as
 \be
 [\theta \wedge \theta](X,Y)=2[\theta(X),\theta(Y)]=2(\theta\wedge \theta)(X,Y).
 \ee
 The left Maurer-Cartan form $\theta$ satisfies the Maurer-Cartan 
 equation:
 \be
 d\theta+\frac{1}{2}[\theta\wedge \theta]=
 d\theta+\theta \wedge \theta=0.
 \ee

\section{Affine harmonic maps}
\label{sc:affineharmonic}
\subsection{Affine harmonic maps}
 Let $(M,g)$ be a Riemannian manifold and $(N,\nabla)$
 an \textit{affine manifold}, 
 that is, a manifold with an affine connection $\nabla$. 
 We denote by $\nabla^M$ and $T$ the Levi-Civita connection 
 of $(M,g)$ and the 
 torsion of the connection $\nabla$ on $N$, respectively.
 Let $\varphi:M\to N$ be a smooth map. 
Then 
 the connection $\nabla$
 induces a unique connection $\nabla^{\varphi}$ 
 on the pull-back tangent bundle $\varphi^{*}TN$ 
 which satisfies the condition
\be
\nabla^{\varphi}_{X}(V\circ{\varphi})=
(\nabla_{d\varphi(X)}V)\circ{\varphi}
\ee
 for any sections $X, Y  \in \varGamma(TM)$
 and $V\in \varGamma(TN)$, see \cite[p.~4]{EL}. 
 The differential $d\varphi$ is 
 naturally  interpreted as a $\varphi^{*}TN$-valued $1$-form. 
 
 The \textit{exterior covariant derivative} of $d\varphi$ is given by
\be
 d^{\nabla^\varphi}d\varphi(X,Y)=\varphi^{*}T(X,Y).
\ee
 Hence $d\varphi$ is a closed $\varphi^{*}TN$-valued $1$-form 
 if and only if $\varphi^{*}T=0$.
 The \textit{second fundamental form} $\nabla{d}\varphi$ 
 of $\varphi$ (with respect to $\nabla$) is 
 a section of $T^* M \otimes T^* M  \otimes \varphi^* TN$
 defined by
\begin{equation}
\nabla{d}\varphi (Y; X):=
(\nabla^{\varphi}_{X}d\varphi)Y
=\nabla^{\varphi}_{X}{d}\varphi(Y)-{d}\varphi(\nabla^{M}_{X}Y),
\ \
X,Y \in \varGamma (TM).
\end{equation}
 One can see that $\nabla d\varphi$ 
 is symmetric if and only if $\varphi^{*}T=0$.
 The \textit{tension field} $\tau(\varphi, \nabla)$ of $\varphi$ 
 is the  section of 
 $\varphi^{*}TN$ defined by 
\be
 \tau(\varphi, \nabla) := \trace (\nabla{d}\varphi).
\ee
 Here $\trace$ denotes the trace with respect to $g$.
 We now arrive at the following definition.
\begin{Definition}
\mbox{}
\begin{enumerate}
\item  A smooth map $\varphi:(M,g)\to (N,\nabla)$ from 
 a Riemannian manifold $(M,g)$
 into an affine manifold $(N,\nabla)$ is said to be an 
 \textit{affine harmonic map} or to be 
 $\nabla$-\textit{harmonic map}
 if 
 \be
 \tau(\varphi, \nabla)=0.
 \ee
 
\item Let us define an affine connection 
 ${}^\dag\nabla$ on an affine manifold $(N,\nabla)$ as follows:
 \be
 {}^\dag\nabla_{X}Y:=\nabla_{X}Y-\frac{1}{2}T(X,Y),\;\;
 X, Y \in \varGamma(TN).
 \ee
 Then ${}^\dag\nabla$ is torsion free and is called the 
 \textit{associated torsion free connection} of $\nabla$.
\end{enumerate}
\end{Definition}
 Since the torsion $T$ is skew-symmetric, it is easy to see 
 that the tension fields 
 $\tau(\varphi, \nabla)$ and 
 $\tau(\varphi, {}^{\dag}\nabla)$ are equal.
 Hence we have the following.
\begin{Proposition}\label{associatedtorsionfree}
 Let $\varphi:(M,g)\to (N,\nabla)$ be a smooth 
 map of a Riemannian manifold into an affine manifold. Then 
 $\varphi$ is $\nabla$-harmonic if and only if 
 it is ${}^\dag\nabla$-harmonic.
\end{Proposition}
\begin{Remark}\label{rm:Harmonic}
\mbox{}
\begin{enumerate}
\item In case $\dim M=2$, the 
 harmonic map equation $\tau(\varphi, \nabla)=0$ is 
 invariant under conformal changes of the metric of $M$, 
 see \cite{EL}.
 Thus the affine-harmonicity makes 
 sense for maps from a Riemann surface.
 Thus if $M$ is orientable, then we can give $M$ 
 a complex structure such that the metric is isothermal and 
 it suffices to consider harmonic maps from this setting. 
 If $M$ is not orientable, we consider the double cover.
\item 
 If $(M,g)$ and $(N,h)$ are Riemannian manifolds and if $\nabla$ 
 denotes the Levi-Civita connection of $h$, 
 then the $\nabla$-harmonicity of $\varphi:M\to N$ coincides 
 with the notion of a harmonic map in the classical sense, 
 that is, it is a critical point of the \textit{energy functional}, 
 \cite{EL}: 
\be
 E(\varphi)=\int_{M}\frac{1}{2} |d\varphi|^{2}\>dv_g.
\ee
\end{enumerate}
\end{Remark}
\subsection{Torsion-free affine harmonic maps}
 More generally, we can introduce the notion of
 an affine $(1, 1)$-harmonic map on a complex manifold 
 in the following manner. 
 Let $(M,J)$ be a complex manifold and
 $\varphi:M\to (N,\nabla)$ a smooth map into an affine 
 manifold.
 With respect to the complex structure $J$, we decompose 
 the complexified tangent bundle $T^{\mathbb{C}}M$ into the 
 Whitney sum $T^{\mathbb C}M=T^{(1,0)}M\oplus T^{(0,1)}M$.
 By restricting the differential $d\varphi$ to 
 $T^{(1,0)}M$ and $T^{(0,1)}M$, we obtain vector bundle 
 morphisms 
 \be
 \partial \varphi:T^{(1,0)}M\to \varphi^{*}TN^{\mathbb C},
 \ \
 \bar{\partial}\varphi: T^{(0,1)}M\to \varphi^{*}TN^{\mathbb C}.
 \ee
 Then the $(0,1)$-covariant derivative 
 $\nabla^{\prime\prime}\partial \varphi$ of 
 $\partial \varphi$ (also called the \textit{Levi-form} of $\varphi$) 
 is defined by
 \be
 (\nabla^{\prime\prime}\partial \varphi)(Z;\overline{W}):=
 (\nabla^{\prime\prime}_{\overline W}\partial \varphi)Z=
 \nabla^{\varphi}_{\overline W}\partial\varphi(Z)-
 \partial\varphi(\bar{\partial}_{\overline W}Z),
 \ \ Z, W
 \in \varGamma(T^{(1,0)}M),
 \ee
 where $\bar \partial$ denotes the $\bar \partial$-operator on 
 $T^{(1, 0)} M$.
\begin{Definition}
 A smooth map $\varphi:(M,J)\to (N,\nabla)$ is said to be 
 an  \textit{affine $(1,1)$-harmonic} 
 if  $\nabla^{\prime\prime}\partial \varphi=0$.
\end{Definition}
 In case $M$ is a Riemann surface, then  affine 
 $(1,1)$-harmonicity
 is equivalent to  affine harmonicity under 
 the {\it torsion free condition}. More precisely, 
 the following relation holds (also see  Remark \ref{rm:torison-free}).
\begin{Proposition}[Theorem 5.1.1 in \cite{Khemar}]
\label{prop:stronglyKehmarcharact}
 Let $M$ be a Riemann surface and $\varphi:M\to (N,\nabla)$ a
 smooth map into an affine manifold. 
 Then $\varphi$ is affine $(1,1)$-harmonic if and only if 
 $\varphi$ is $\nabla$-harmonic and $\varphi^{*}T=0$.
\end{Proposition}
\begin{proof}[Sketch of proof]
Take $(1,0)$ vector fields $Z$ and $W$.
Then $Z$ and $W$ are represented as
\be
Z=X-\sqrt{-1}JX, \ \ W=Y-\sqrt{-1}JY,\ \ 
X,Y\in \varGamma(TM).
\ee
The $(0,1)$-covariant derivative of 
$\partial\varphi$ is computed as
\begin{align*}
(\nabla^{\prime\prime}\partial \varphi)(Z;\overline{W})
=&
\{
(\nabla d\varphi)(X,Y)+(\nabla d\varphi)(JX, JY)
\}
\\
&
+ \sqrt{-1}
\{
(\nabla d\varphi)(X,JY)-(\nabla d\varphi)(JX, Y)
\}
\end{align*}
 This equation implies  
 $\varphi^{*}T=0$ and $\tau(\varphi)=0$. 
 More precisely, take 
 a Hermitian metric $g$ in the conformal class.
 Then for  a local orthonormal frame field $\{e_1,e_2=Je_1\}$
 we choose 
 $X=Y=e_1.$ Then the real part of the 
 equation $(\nabla^{\prime\prime}\partial \varphi)(Z;\overline{W})
 =0$ implies $\tau(\varphi , \nabla) = \trace (\nabla d \varphi)=0$. 
 The imaginary part implies 
 $\varphi^{*}T(e_1,e_2)=0$. Hence 
 $\varphi^{*}T=0$.
 \end{proof}
 Based on this characterization we give the following definition:
\begin{Definition}\label{def:stronglyharmonic}
 Let $\varphi:M\to (N,\nabla)$ be a smooth map of a Riemann surface 
 into an affine manifold. Then $\varphi$ is said to be 
 \textit{torsion-free affine harmonic} or
 \textit{torsion-free $\nabla$-harmonic}  if $\varphi$ is 
 $\nabla$-harmonic and $\varphi^{*}T=0$, that is, 
 it is an affine $(1,1)$-harmonic map on a Riemann surface. 
\end{Definition}
\begin{Remark}
\mbox{}
\begin{enumerate}
\item
 Torsion-free harmonic maps were first observed by 
 Burstall and Pedit \cite{BP}. 
 In \cite{Khemar},  torsion-free $\nabla$-harmonic maps were 
 called ``strongly harmonic maps''.

\item For a Riemann surface $M$ and a Riemannian manifold 
 $(N,h)$ with Levi-Civita connection $\nabla$ of $h$,
 the affine $(1, 1)$-harmonicity of $\varphi : M \to N$
 coincides with the harmonicity in the classical sense, 
 see \cite{EL} and (2) in Remark \ref{rm:Harmonic}.
 \end{enumerate}
\end{Remark}
\section{Affine harmonic maps into Lie groups}\label{sc:affineLiegroup}
 In this section we discuss affine harmonic maps 
 from a Riemann surface into Lie groups.
 Let $G$ be a connected real Lie group
 and denote by $\mathfrak{g}$ 
 the Lie algebra of $G$, that is, 
 the tangent space of $G$ at the unit element $\id \in G$. 
 Hereafter we restrict our attention to left-invariant affine 
 connections on $G$. 
\subsection{Left-invariant connections} 
 Take a bilinear map 
 $\mu:\mathfrak{g}\times \mathfrak{g}\to \mathfrak{g}$.
 Then we can define a left-invariant affine 
 connection $\nm$ on $G$ by its value at the unit element $\id \in G$ by
\be
\nm_{X}Y=\mu(X,Y), \ \ X,Y
\in \mathfrak{g}.
\ee
 From \cite{Nomizu} we know that all left-invariant affine connections are obtained in this 
 way.
\begin{Proposition}[\cite{Nomizu}]\label{allconnections}
 Let $m_\mathfrak{g}$ be the vector 
 space of all $\mathfrak{g}$-valued 
 bilinear maps on $\mathfrak{g}$ and 
 $a_G$ the affine space of all 
 left-invariant affine connections on $G$.
 Then the map
\be
 m_{\mathfrak g}\ni \mu\longmapsto \nm \in a_G 
\ee
 is a bijection between $m_{\mathfrak g}$ 
 and $a_G$. The torsion $T^\mu$ of $\nm$ 
 is given by
\begin{equation}\label{eq:torsionformu}
T^{\mu}(X,Y)=-[X,Y]+\mu(X,Y)-\mu(Y,X)
\end{equation}
 for all $X$, $Y\in \mathfrak{g}$.
\end{Proposition}
\begin{Definition}\label{def:threeconnections}
 We define a one parameter family 
 $\{ {}^{(t)}\nabla \;| \;t\in \mathbb{R}\}$ of 
 bi-invariant connections  $\nt:=\nabla^{\mu(t)}$ by
 $\mu(t)={}^{(t)}\mu$, where
\begin{equation}\label{eq:familyconnection}
 {}^{(t)}\mu(X,Y):=\frac{1}{2}(1+t) [X, Y],\ \  X,Y \in \mathfrak{g}.
\end{equation}
 There are three particular connections in 
 the family $\{\nt\; |\;t \in \mathbb R\}$:
\begin{enumerate}
\item  The \textit{canonical connection}:  ${}^{(-1)}\nabla $ \hspace{2mm} defined by setting  $t =- 1$.

\item The \textit{anti-canonical connection}: ${}^{(1)}\nabla $\hspace{2mm} defined 
 by setting $t =1$. 

\item The \textit{neutral connection}: ${}^{(0)}\nabla $ \hspace{2mm}defined by setting  $t =0$.
\end{enumerate} 
\end{Definition}
\begin{Remark} 
\mbox{}
\begin{enumerate}
\item
 The canonical connection and the anti-canonical connection have been 
 discussed in \cite{KN2, Agricola, Khemar}.
 In \cite{KN2}, the connections $\can$, $\anti$ 
 and $\neutral$ have been called \textit{Cartan-Schouten's} 
 $(-)$-\textit{connection}, $(+)$-\textit{connection} and $(0)$-\textit{connection}, 
 respectively. 
\item
 The set of all \textit{bi-invariant} connections on $G$ 
 is parametrized by 
\be
m^{bi}_{\mathfrak g}=
\left\{ \mu \in m_{\mathfrak g}\  \vert \ 
\mu(\ad (g)X,\ad (g)Y)=\ad (g)
\mu(X,Y),\ \ \textrm{for any}\ X,Y\in \mathfrak{g}, 
\ g\in G\right\}.
\ee
 Laquer \cite{Laquer} proved that for any compact simple 
 Lie group $G$, the set 
 $m^{bi}_{\mathfrak g}$ is $1$-dimensional except 
 for the case $G=\mathrm{SU}_n$ with $n\geq 3$. 
 More precisely, for any compact Lie group $G$, except for $ G =\mathrm{SU}_n$ ($n\geq 3$), 
 the set $m^{bi}_{\mathfrak g}$ is given by 
\be
 m^{bi}_{\mathfrak g}=\{ {}^{(t)}\mu \  \vert \  t \in \mathbb R \},
\ee
 where ${}^{(t)}\mu$ is defined in \eqref{eq:familyconnection}.
 Thus the corresponding set of bi-invariant connections
 is given by $a^{bi}_{G}=\{{}^{(t)}\nabla\ \vert \ t\in \mathbb{R}\}$.
 In case $G=\mathrm{SU}_n$ with $n\geq 3$, 
 $m^{bi}_{\mathfrak{su}_n}$ 
 is $2$-dimensional. The set $m_{\mathfrak{su}_n}^{bi}$ is parametrized by
\be
 m_{\mathfrak{su}_n}^{bi}
=\{{}^{(t,s)}\mu\ \vert \ t,s\in \mathbb{R}\}
\ee
with
${}^{(t,s)}\mu(X,Y)=
\frac{1}{2}(1+t)[X,Y]
+\sqrt{-1}\; s\left(
(XY+YX)-\frac{2}{n}\mathrm{tr}\>(XY)\id
\right)$.
\end{enumerate}
\end{Remark}

 We note the following elementary, but important Lemma.
\begin{Lemma}\label{lem:skewconnection}
\mbox{}
\begin{enumerate}
\item The bi-invariant connection 
 $\nt$ is torsion-free if and only if it is the neutral connection $\neutral$.

\item Assume that $\mu$ is skew-symmetric, then the 
 torsion free connection ${}^\dag\nm$ 
 associated to $\nm$ coincides with $\neutral$. 

\item Let $M$ be a Riemann surface, $G$ a Lie group  and $\nabla^{\mu}$ a 
connection on $G$ with $\mu$ skew-symmetric. Then a map
$\varphi:M\to G$ is $\nabla^{\mu}$-harmonic if and only if it is 
$\neutral$-harmonic.
\end{enumerate}
\end{Lemma}

\begin{proof} 
 (1):  It is straightforward to verify that using \eqref{eq:torsionformu}
 the torsion of $\nt$ is given by $T ^{{}^{(t)}\mu} (X,Y) = t [X,Y]$.
 Thus the claim immediately follows.

 (2): Assume that $\mu$ is skew-symmetric, then
 $T^{\mu}(X,Y)=-[X,Y]+2\mu(X,Y)$
 and the associated connection ${}^\dag\nabla^\mu$
 is computed as
\be
{}^\dag\nabla^\mu_{X}Y=
 \nabla^\mu_{X}Y-\frac{1}{2}T^{\mu}(X,Y)
 = \mu(X,Y)-\frac{1}{2}
 \left(
 -[X,Y]+2\mu(X,Y)
 \right)=\frac{1}{2}[X,Y]
\ee
 for all $X$, $Y\in \mathfrak{g}$. Thus ${}^\dag \nm = \neutral$.

(3): Follows immediately from  Proposition \ref{associatedtorsionfree}.
\end{proof}

\subsection{Affine harmonic maps}
 Now let $\varphi:M\to G$ be a smooth map of a 
 Riemann surface $M$ into a Lie group $G$. 
 Then the pull-back $1$-form $\alpha:=\varphi^{*}\theta$ of the left Maurer-Cartan 
 form $\theta$ by $\varphi$ satisfies the \textit{Maurer-Cartan equation}:
\be
 d\alpha+\frac{1}{2}[\alpha \wedge \alpha]=0.
\ee
 With respect to the conformal structure of $M$, 
 we decompose $\alpha$ as 
 $\alpha=\alpha^{\prime}+\alpha^{\prime\prime}$. 
 Then the Maurer-Cartan equation is rephrased as
\begin{equation}\label{eq:MCforalpha}
 \bar{\partial}\alpha^{\prime}+\partial\alpha^{\prime\prime}+
 [\alpha^{\prime}\wedge \alpha^{\prime\prime}]=0.
\end{equation}
\begin{Remark}
 When $G$ is a linear Lie group, then $\alpha=\varphi^{*}\theta$ has the form
  $\alpha=\varphi^{-1}d\varphi$. Even if $G$ is not a 
 linear Lie group (\textit{e.g.}, the universal covering 
 $\widetilde{\mathrm{SL}_2\mathbb{R}}$ of $\mathrm{SL}_2\mathbb{R}$),
 we can represent $\varphi^{*}\theta$ as $\varphi^{-1}d\varphi$ by using 
 the tangent group structure of $G$, \cite{Kobayashi}.
\end{Remark}
 From now on we equip $G$ with a left-invariant connection $\nabla^\mu$ and 
 take a Hermitian metric $g$ in the conformal class of $M$. Then 
 the second fundamental form $\nm {d}\varphi$ with 
 respect to $\nabla^\mu$ is 
 related to the second fundamental 
 form $\can {d} \varphi$ with 
 respect to the 
 canonical connection by
\be
\theta((\nm_{X} d\varphi)Y)= 
\theta((\can_{X} {d}\varphi)Y)
+\mu(\alpha(X),\alpha(Y)),
\ee
 where $\alpha =\varphi^* \theta$.
 Hence the tension field $\tau(\varphi,\nm)$ 
 with respect to $\nabla^\mu$ is given by
\be
 \theta(\tau(\varphi,\nm))=\theta(\tau(\varphi,\can ))
 +\trace \mu(\alpha,\alpha).
\ee
\begin{Proposition}\label{prop:affine-harmonicity}
 Let $\varphi:M \to (G,\nm)$ be a 
 smooth map of a Riemann surface into 
 a Lie group and $\alpha= \varphi^{*}\theta=\alpha^{\prime} +
 \alpha^{\prime \prime}$ the Maurer-Cartan form. Then $\varphi$ 
 is $\nm$-harmonic if and only if $\alpha$ satisfies
\begin{equation}\label{eq:mu-harmonic}
\bar{\partial} \alpha^{\prime}- \partial \alpha^{\prime \prime}
+ 2 (\sym \mu) 
 (\alpha^{\prime \prime} \wedge \alpha^{\prime})=0.
\end{equation}
\end{Proposition}
\begin{proof}

 Take a Hermitian metric $g = e^u dz d \bar z$ with 
 a conformal coordinate $z$ in the 
 conformal class of $M$. With respect to this metric $g$, 
 since $* dz = - \sqrt{-1} dz$, $* dz = \sqrt{-1} d\bar z$ 
 and $*(d z \wedge d \bar z) = - 2 \sqrt{-1} e^{-u}$, we obtain
\be
\theta(\tau_{g} (\varphi,\can))
= * d * \alpha=
-\sqrt{-1}* (\bar{\partial}\alpha^{\prime}
-\partial \alpha^{\prime\prime}).
\ee
 On the other hand, 
\be
\trace \mu(\alpha,\alpha)
=- 2\sqrt{-1}* (\sym \mu)(\alpha^{\prime \prime}
\wedge \alpha^{\prime}).
\ee
Hence we obtain
\be
\theta(\tau_g(\varphi,\nm))
= -
\sqrt{-1}*
\left\{
\bar{\partial}\alpha^{\prime}-
\partial\alpha^{\prime\prime}
+2(\sym \mu)(\alpha^{\prime \prime}
\wedge \alpha^{\prime})
\right\}.
\ee
Thus the claim follows.

\end{proof}

\begin{Corollary}\label{coro:affharmonic}
 Let $\varphi:M \to (G,\nm)$ be a 
 smooth map of a Riemann surface into 
 a Lie group and $\alpha = \varphi^{*}\theta$ 
 the Maurer-Cartan form of $\varphi$. 
 Then the following properties are mutually equivalent:
\begin{enumerate}
\item The map $\varphi$ is $\nm$-harmonic.
\item The Maurer-Cartan form $\alpha$ satisfies
$ d * \alpha +(\sym \mu)(*\alpha \wedge \alpha)=0$, 
or equivalently, for $\alpha =\alpha^{\prime} + \alpha^{\prime \prime}$,
\begin{equation}
\label{0-HME}
 \bar{\partial} \alpha^{\prime}- \partial\alpha^{\prime \prime} 
 + 2 (\sym \mu)(\alpha^{\prime \prime} \wedge \alpha^{\prime})=0.
\end{equation}
\item 
 The Maurer-Cartan form 
 $\alpha$ satisfies
 $d \alpha + d * \alpha +\frac{1}{2}[\alpha \wedge \alpha] + (\sym \mu)(*\alpha \wedge \alpha)=0$, or equivalently,
 for $\alpha = \alpha^{\prime} + \alpha^{\prime \prime}$,
\begin{equation}\label{0-pluri}
 2\bar{\partial} \alpha^{\prime}+
 [\alpha^{\prime}\wedge \alpha^{\prime \prime}]
 +2 (\sym \mu)(\alpha^{\prime \prime} \wedge \alpha^{\prime})=0.
\end{equation}
\end{enumerate}
\end{Corollary}

\begin{proof}
 The equivalence of $(1)$ and $(2)$ has been already shown.
 To show the equivalence of $(2)$ and $(3)$, we add the Maurer-Cartan equation 
 \eqref{eq:MCforalpha} and the harmonicity equation \eqref{0-HME}.
 Then the equation \eqref{0-pluri} follows. Conversely, we subtract  
 the  Maurer-Cartan equation \eqref{eq:MCforalpha} from  the equation \eqref{0-pluri}, 
 then the harmonicity equation \eqref{0-HME} follows.
\end{proof}
 Then the following fundamental theorem is obtained.
 \begin{Theorem}\label{IntegrabilityTheorem}
 Let $\nm$ be a left-invariant affine 
 connection on a Lie group $G$ determined by a bilinear map 
 $\mu$ and $\varphi:M\to G$ a $\nm$-harmonic map
 from a Riemann surface $M$ into $G$. Then 
 $\alpha=\varphi^{*}\theta=\alpha^{\prime}+ \alpha^{\prime \prime}$ 
 satisfies \eqref{0-pluri}.
 Conversely, let $\mathbb{D} \subset \mathbb C$ be a simply-connected 
 domain and $\alpha=\alpha^{\prime}+\alpha^{\prime \prime}$ a 
 $\mathfrak{g}$-valued $1$-form satisfying \eqref{0-pluri}.
 Then there exist a $\nm$-harmonic map
 $\varphi:\mathbb{D}\to G$ such that $\varphi^{*}\theta=\alpha$. 
\end{Theorem}
\begin{proof}
We only need to prove the converse statement.
 Let $\alpha = \alpha^{\prime}+\alpha^{\prime \prime}$ be a
  $\mathfrak g$-valued $1$-form on $\mathbb{D}$ satisfying \eqref{0-pluri}. 
 Then subtraction and addition of the complex conjugate 
 of \eqref{0-pluri} to itself
 gives the integrability condition \eqref{eq:MCforalpha} 
 and the harmonicity condition \eqref{0-HME}, respectively.
\end{proof}
\subsection{Torsion-free affine harmonic maps}
 In this subsection, we compute the \textit{torsion-free} $\nm$-harmonicity
 condition for a smooth map $\varphi:M\to (G,\nm)$ in terms of 
 the Maurer-Cartan form $\alpha=\varphi^{*}\theta$.
 Moreover, we relate  $\nm$-harmonicity 
 to torsion-free $\nm$-harmonicity.

 By  Proposition  \ref{allconnections}, the torsion $T^{\mu}$ along 
 $\varphi$ is given by 
\begin{equation}\label{eq:Torsion}
 \varphi^* T^{\mu}(\alpha^{\prime} \wedge \alpha^{\prime \prime})
 =-[\alpha^{\prime} \wedge \alpha^{\prime \prime}] + 
 2 (\sk \mu )(\alpha^{\prime} \wedge \alpha^{\prime \prime}).
\end{equation}
 Then we have the following theorem.
\begin{Theorem}\label{thm:stronglyharmonic}
 Let  $\varphi:M\to (G,\nm)$ be a smooth map.
 Then the following statements are equivalent$:$
\begin{enumerate}
 \item $\varphi$ is a torsion free  $\nm$-harmonic map.

 \item The Maurer-Cartan form $\alpha = \varphi ^* \theta$ satisfies
 \begin{align}
 (\sk \mu)(\alpha^\prime\wedge \alpha^{\prime\prime})
 -\frac{1}{2}[\alpha^\prime\wedge \alpha^{\prime\prime}]&=0, \label{T}\\
 \bar{\partial} \alpha^{\prime}+
 \mu(\alpha^{\prime\prime}\wedge \alpha^{\prime})&=0. \label{eq:stronglymuharmonic}
\end{align}
\end{enumerate}
\end{Theorem}
\begin{proof} 

 $(1) \Rightarrow (2)$:
 Since the torsion $T^{\mu}$ 
 in \eqref{eq:Torsion} is assumed to  vanish, we obtain \eqref{T}.
 Moreover, since $(\sk \mu) (\alpha^{\prime} \wedge \alpha^{\prime \prime}) = 
 (\sk \mu) (\alpha^{\prime \prime} \wedge \alpha^{\prime})$ and 
 the $\nm$-harmonicity is characterized by the equation 
 \eqref{0-pluri}, we conclude the equation \eqref{eq:stronglymuharmonic}, 
 proving $(2)$.

 $(2) \Rightarrow (1)$: Starting from \eqref{eq:stronglymuharmonic} 
 we can rephrase it in the form  
 \begin{equation} \label{**}
 \bar \partial \alpha^{\prime} 
 - (\sym \mu)(\alpha^{\prime} \wedge \alpha^{\prime \prime}) + 
 (\sk \mu )(\alpha^{\prime} \wedge \alpha^{\prime \prime}) =0.
 \end{equation}
 Comparing this equation to \eqref{T}, we obtain the $\nm$-harmonicity equation 
 \eqref{0-pluri}. Moreover the equation \eqref{T} is just the torsion-free condition, 
 whence $(1)$.

\end{proof}
\begin{Remark}\label{rm:torison-free}
 Theorem \ref{thm:stronglyharmonic} 
 can be obtained by computing the Levi-form of the 
 map $\varphi$. From \cite{BR}, 
 the canonical connection and the left 
 Maurer-Cartan form satisfy the following relation:
\be
 \theta (\can_X Y) = X \theta(Y)- [ \theta (X), \theta (Y)].
\ee
 Since $\nm = \can + \mu$, we have
 $\theta (\nm_X Y) = X \theta(Y)- [ \theta (X), \theta (Y)] + \mu (\theta (X), 
 \theta (Y))$.
 Setting $\varphi^* \theta = \alpha$, we have   
\begin{align*}
 \theta \left({\nabla^{ \mu}}^{\prime \prime}
 \partial \varphi \right) & = - \sqrt{-1} * \left\{
 \bar \partial \alpha^{\prime}
 - [\alpha^{\prime \prime} \wedge \alpha^{\prime}] 
 + \mu (\alpha^{\prime \prime} \wedge \alpha^{\prime})\right\}. 
\end{align*}
 Thus ${\nabla^{ \mu}}^{\prime \prime}
 \partial \varphi =0$ if and only if 
 \begin{equation}\label{eq:torsion-free}
 \bar \partial \alpha^{\prime}
 - [\alpha^{\prime \prime} \wedge \alpha^{\prime}] 
 + \mu (\alpha^{\prime \prime} \wedge \alpha^{\prime}) =0.
 \end{equation}
 Adding and subtracting the complex conjugate of \eqref{eq:torsion-free} to itself, 
 we observe that 
 \eqref{eq:torsion-free} is equivalent to the two equations
 \begin{align*}
  2 (\sk \mu) (\alpha^{\prime} \wedge \alpha^{\prime \prime}) - 
 [\alpha^{\prime} \wedge \alpha^{\prime \prime}] =0,\\ 
  \bar \partial \alpha^{\prime} -
  \partial \alpha^{\prime \prime} 
 + 2 (\sym \mu) (\alpha^{\prime \prime} \wedge \alpha^{\prime}) =0,
 \end{align*}
 which are equivalent with the conditions \eqref{T} and 
\eqref{eq:stronglymuharmonic}.
 We note that we have used here $\mu = \sym \mu+ \sk \mu$,  
 $(\sym \mu) (\alpha^{\prime \prime} \wedge \alpha^{\prime})  = -
  (\sym \mu) (\alpha^{\prime} \wedge \alpha^{\prime \prime})$,
 $(\sk \mu) (\alpha^{\prime \prime} \wedge \alpha^{\prime})  = 
  (\sk \mu) (\alpha^{\prime} \wedge \alpha^{\prime \prime})$ 
 and $[\alpha^{\prime \prime} \wedge \alpha^{\prime}] =
 [\alpha^{\prime} \wedge \alpha^{\prime \prime}]$.
\end{Remark}
 Among the three connections $\can, \anti$ and $\neutral$ in 
 Definition \ref{def:threeconnections}, 
 the $\neutral$ connection automatically 
 satisfies the conditions of the theorem above. Thus we have the following.
\begin{Corollary}
 A smooth map $\varphi:M\to(G, \neutral)$ 
 is torsion-free $\neutral$-harmonic if and only if 
 it is $\neutral$-harmonic.
\end{Corollary}
\subsection{Zero-curvature representations for admissible affine harmonic maps}
 In this subsection, we introduce {\it admissible} $\nm$-harmonic 
 maps and show that admissible $\nm$-harmonic maps admit a zero-curvature representation.
 As a corollary we obtain immediately the important fact that all $\nabla^{\mu}$-harmonic 
 maps admit a zero-curvature representation if $\mu$ is skew-symmetric.
 This applies in particular to the connections ${}^{(t)}\nabla$.
 We first introduce the following definition.
\begin{Definition}
 A $\nm$-harmonic map $\varphi$ from a Riemann surface $M$ 
 into a Lie group $G$ with the Maurer-Cartan form $\alpha = \varphi^* \theta =
 \alpha^{\prime} + \alpha^{\prime \prime}$ satisfying the 
 condition 
 \begin{equation}\label{eq:admissible}
 (\sym \mu)(*\alpha \wedge \alpha)=0,
\;\;\mbox{or equivalently,} \;\;
 (\sym \mu)(\alpha^{\prime} \wedge \alpha^{\prime \prime})=0,
 \end{equation}
 is called an {\it admissible $\nm$-harmonic map}.
\end{Definition}
\begin{Remark}\label{Rm:admissible}
 If $\mu$ is skew-symmetric, then \eqref{eq:admissible} is automatically satisfied. 
 This applies, in particular, to the connections  $\nt, \; (t\in \mathbb R)$.
 Thus a smooth map $\varphi : M \to (G, \nt)$ is admissible $\nt$-harmonic if and 
 only if it is $\nt$-harmonic.
\end{Remark}
 We now characterize an admissible $\nm$-harmonic map into a Lie group 
 in terms of a family of flat connections.
\begin{Proposition}\label{prop:admissibleharmonic}
 Let $\varphi:M\to (G, \nm)$ be an admissible $\nm$-harmonic map into a 
 Lie group $G$ equipped with a left-invariant connection $\nabla^\mu$
 and $\alpha = \varphi^{*}\theta = \alpha^{\prime}+
 \alpha^{\prime \prime}$ the Maurer-Cartan form of $\varphi$.
 Then the loop of connections $d+\alpha_{\lambda}$ defined by
\begin{equation}\label{alpha-family}
 \alpha_{\lambda}:=
\frac{1}{2}(1-\lambda^{-1})\alpha^{\prime} 
 +\frac{1}{2}(1-\lambda)\alpha^{\prime\prime}
\end{equation}
 is flat for all $\lambda \in \mathbb S^1$. 

 Conversely assume that $\mathbb{D} \subset \mathbb C$ is simply-connected
 and let $\alpha_{\lambda}=
 \frac{1}{2}(1-\lambda^{-1})\alpha^{\prime}+
 \frac{1}{2}(1-\lambda)\alpha^{\prime \prime}$ 
 be an $\mathbb{S}^1$-family of 
 $\mathfrak{g}$-valued $1$-forms satisfying 
 the condition \eqref{eq:admissible} and 
 the zero-curvature condition{\rm:}
\be
 d\alpha_{\lambda}+\frac{1}{2}[\alpha_\lambda\wedge \alpha_\lambda]=0
\ee
 for all $\lambda \in \mathbb{S}^1$.
 Then there exists a $1$-parameter family of maps $F_\lambda:
 \mathbb{D}\times \mathbb{S}^1\to G$ such that 
 $F_{\lambda}^{*}\theta= \alpha_{\lambda}$. 
 The map $F_{\lambda}$ is called 
 the \textrm{extended solution associated with $\alpha$ and it is}
 \textrm{admissible} $\nm$-harmonic for $\lambda=\pm 1$. 
\end{Proposition}
\begin{proof}
 Let $\varphi:\mathbb{D}\to G$ be a smooth map and define the 
 loop $\{d+\alpha_\lambda\}$ of connections by \eqref{alpha-family}. 
 Then a direct computation shows that $2 d\alpha_{\lambda}+[\alpha_{\lambda}\wedge \alpha_{\lambda}]$
 is equal to 
\be
 \bar{\partial}\alpha^{\prime}+\partial\alpha^{\prime
\prime}+[\alpha^{\prime}\wedge \alpha^{\prime\prime}]
-\frac{\lambda^{-1}}{2}\left(2 \bar{\partial}\alpha^{\prime}+
[\alpha^{\prime}\wedge \alpha^{\prime\prime}]\right)
-\frac{\lambda}{2} 
\left(2 \partial\alpha^{\prime\prime}+
[\alpha^{\prime\prime}\wedge \alpha^{\prime}]\right).
\ee
 Assume that $\varphi$ is a $\nm$-harmonic map with condition \eqref{eq:admissible}, 
 then we have the Maurer-Cartan equation \eqref{eq:MCforalpha} 
 and the harmonicity equation \eqref{0-HME}.
 Hence $d+\alpha_{\lambda}$ is flat for all $\lambda \in \mathbb S^1$.

 Conversely, let $\alpha_{\lambda}$ be an 
 $\mathbb{S}^1$-family of $\mathfrak{g}$-valued
 $1$-forms of the form \eqref{alpha-family} defined on a 
 simply-connected $\mathbb{D}$. 
 Assume that $d+\alpha_{\lambda}$ is 
 flat for all $\lambda \in \mathbb{S}^1$. 
 Then there exists a $1$-parameter family of maps
 $F_{\lambda}:\mathbb{D}\times \mathbb{S}^1\to G$ such that 
 $F_{\lambda}^{*}\theta=\alpha_\lambda$. 
 The flatness of all $d+\alpha_\lambda$ implies
\be
 \bar{\partial}\alpha^{\prime}+\partial\alpha^{\prime \prime} 
 +[\alpha^{\prime} \wedge \alpha^{\prime \prime}]=0,
\ \
\bar \partial \alpha^{\prime}+\frac{1}{2}[\alpha^{\prime \prime} 
 \wedge \alpha^{\prime}]=0,
\ \
\partial \alpha^{\prime \prime}+\frac{1}{2}[\alpha^{\prime} 
 \wedge \alpha^{\prime \prime}]=0.
\ee
 The last two equations together with the condition \eqref{eq:admissible} 
 are nothing but the admissible $\nm$-harmonic map equation for $F_{\lambda =\pm 1}$.
\end{proof}
\begin{Remark}
\mbox{}
\begin{enumerate}
 \item It is easy to see, since $\alpha_{\lambda =1} =0$, 
 that the map $F_{\lambda =1}$ is a constant map.

\item Take a smooth curve $\gamma:\mathbb{S}^1\to G$ 
 satisfying $\gamma(1)=\id$.
 Then there exists a unique map $F_\lambda :\mathbb{D}\times 
 \mathbb{S}^1\to G$ satisfying $F_{\lambda}^{*}\theta
 =\alpha_{\lambda}$ and 
 the initial condition $F_{\lambda}(z_0)=\gamma(\lambda)$.

\item  Clearly, when assuming $F_{\lambda=1} = \id$ we obtain a map
 from $\mathbb{D}$ into the \textit{based loop group}
\be
 \Omega{G}=\{\gamma:\mathbb{S}^1\to G\ \vert \ \gamma(1)=\id \>\}.
\ee
 Note that extended solutions are holomorphic curves into 
 $\Omega G$ with respect to the canonical complex structure 
 of $\Omega G$.
\end{enumerate}
\end{Remark}
 Since $\nm$-harmonic maps with skew-symmetric $\mu$ 
 automatically satisfy 
 the admissibility condition 
 \eqref{eq:admissible}, by Remark \ref{Rm:admissible} and 
 by $(3)$ of Lemma \ref{lem:skewconnection}, $\nm$-harmonicity is equivalent 
 to $\neutral$-harmonicity. 
 We thus  have the following theorem as a corollary 
 to Proposition \ref{prop:admissibleharmonic}.
\begin{Theorem}\label{thm:familyofconnections}
 Let $M$ be a Riemann surface, $G$ a Lie group, $\nm$ 
 a left-invariant connection with skew-symmetric $\mu$ and 
 $\neutral$ the bi-invariant connection stated Definition \ref{def:threeconnections}.
 Then a map $\varphi : M \to G$ is a $\nm$-harmonic map if and only if 
 it is a $\neutral$-harmonic map. Moreover 
 let $\alpha = \varphi^{*}\theta = \alpha^{\prime} + 
 \alpha^{\prime \prime}$ be the Maurer-Cartan form of $\varphi$.
 Then the loop of connections $d+\alpha_{\lambda}$ defined by
\begin{equation}\label{alpha-family-skew}
 \alpha_{\lambda}:=
\frac{1}{2}(1-\lambda^{-1})\alpha^{\prime} 
 +\frac{1}{2}(1-\lambda)\alpha^{\prime\prime}
\end{equation}
 is flat for all $\lambda \in \mathbb S^1$. 

 Conversely, assume that $\mathbb{D}$ is simply-connected
 and let $\alpha_{\lambda}=
 \frac{1}{2}(1-\lambda^{-1})\alpha^{\prime}+
 \frac{1}{2}(1-\lambda)\alpha^{\prime \prime}$ 
 be an $\mathbb{S}^1$-family of 
 $\mathfrak{g}$-valued $1$-forms satisfying 
 the zero-curvature condition{\rm:}
\be
 d\alpha_{\lambda}+\frac{1}{2}[\alpha_\lambda\wedge \alpha_\lambda]=0
\ee
 for all $\lambda \in \mathbb{S}^1$.
 Then there exists a $1$-parameter family of maps $F_\lambda:
 \mathbb{D}\times \mathbb{S}^1\to G$ such that 
 $F_{\lambda}^{*}\theta= \alpha_{\lambda}$. 
 The map $F_{\lambda}$ is called an \textrm{extended solution} and it is
 $\nm$-harmonic for  any skew-symmetric $\mu$,
 thus it is also $\neutral$-harmonic, for $\lambda=\pm 1$.
\end{Theorem}
\subsection{Zero-curvature representations for torsion-free affine 
 harmonic  maps}\label{subsc:strongly}
 In this subsection, we show that an admissible $\nm$-harmonic 
 map satisfying
 $[\alpha^{\prime} \wedge \alpha^{\prime \prime}] =0$ 
 admits a zero-curvature representation. As a corollary, we show that 
 a torsion-free $\nt$-harmonic $(t \neq 0)$ map
 admits a zero-curvature representation and give an explicit formula for such 
 a map.
\begin{Proposition}\label{prop:stronglyharmonicloop}
 Let $\varphi:M\to (G, \nm)$ be an admissible 
 $\nm$-harmonic map into a Lie group with
  a left-invariant connection $\nabla^\mu$ and 
 $\alpha = \varphi^{*}\theta  = \alpha^{\prime} + 
 \alpha^{\prime \prime}$ the Maurer-Cartan form of $\varphi$.
 Further assume that 
 \begin{equation}\label{eq:strongcondition}
 [\alpha^{\prime} \wedge \alpha^{\prime \prime}] =0.
 \end{equation}
 Then the loop of connections $d+\alpha_{\lambda}$ defined by
\begin{equation}\label{eq:alpha-familystrong}
 \alpha_{\lambda}:= \lambda^{-1}\alpha^{\prime} 
 +\lambda \alpha^{\prime\prime}
\end{equation}
 is flat for all $\lambda \in \mathbb S^1$. 

 Conversely assume that  $\mathbb{D}$ is simply-connected and 
 let $\alpha_{\lambda}= \lambda^{-1}\alpha^{\prime}+ \lambda
 \alpha^{\prime \prime}$ an $\mathbb{S}^1$-family of 
 $\mathfrak{g}$-valued $1$-forms satisfying the admissibility 
 condition \eqref{eq:admissible}
 and the zero-curvature condition{\rm:}
\be
 d\alpha_{\lambda}+\frac{1}{2}[\alpha_\lambda\wedge \alpha_\lambda]=0
\ee
 for all $\lambda \in \mathbb{S}^1$.
 Then there exists a $1$-parameter family of maps $\varphi_\lambda:
 \mathbb{D}\times \mathbb{S}^1\to G$
 such that  $\varphi^{*}\theta= \alpha_{\lambda}$. 
 The maps $\varphi_{\lambda}$ are admissible 
 $\nm$-harmonic maps satisfying the condition \eqref{eq:strongcondition}.
\end{Proposition}
\begin{proof}
 By Proposition \ref{prop:affine-harmonicity},
 admissible $\nm$-harmonicity satisfying 
 condition \eqref{eq:strongcondition}
 gives
\be
 \bar \partial \alpha^{\prime} =
 \partial \alpha^{\prime \prime} = [\alpha^{\prime} \wedge 
 \alpha^{\prime \prime}] =0.
\ee
 Then it is easy to check that the flatness of $d + \alpha_{\l}$
 is exactly equivalent to the above three equations. 
 Conversely, it is also easy to check that the flatness of $d + \alpha_{\l}$ under 
 the conditions \eqref{eq:admissible} 
 gives admissible $\nm$-harmonicity satisfying the condition
 \eqref{eq:strongcondition}.
\end{proof}
 Moreover we can construct all admissible 
 $\nm$-harmonic maps satisfying \eqref{eq:strongcondition} as follows.
\begin{Theorem}\label{thm:strongharmonic}
 Let $G$ be a real Lie group with left-invariant connection $\nm$.
 An admissible $\nm$-harmonic map $\varphi:M\to G$ 
 satisfying \eqref{eq:strongcondition} 
 is, up to left translation, a map of the form
 \begin{equation}\label{eq:varPhi}
\varphi(z, \bar z) =
 \exp\left(\int_{z_*}^z\varPhi(t) dt +\int_{z_*}^{\bar z} \bar{\varPhi}(t) dt\right),
 \end{equation}
 where $z$ is a conformal coordinate on $M$,  
 $\varPhi$ is holomorphic and takes values in $\mathfrak g.$ Moreover,  $\bar\varPhi$, the complex conjugate of $\varPhi$, satisfies 
 \begin{equation}\label{eq:twoconditions}
 [\varPhi,\bar{\varPhi}]=0 \;\;\mbox{and}\;\;
 (\sym \mu)(\varPhi,\bar{\varPhi})=0.
\end{equation}
\end{Theorem} 
\begin{proof}
 We have seen in the proof of Proposition \ref{prop:stronglyharmonicloop} that 
 admissible $\nm$-harmonic maps $\varphi:M\to G$ 
 satisfying \eqref{eq:strongcondition} are characterized by the equations 
 \begin{equation}\label{eq:stronglyharmonic}
 (\sym \mu) (\alpha^{\prime} \wedge \alpha^{\prime \prime}) = 
 [\alpha^{\prime} \wedge \alpha^{\prime \prime}] = 
 \bar \partial \alpha^{\prime} = \partial \alpha^{\prime \prime} =0.
 \end{equation}
 Set $\alpha^{\prime} = \varPhi dz$ and 
 $\alpha^{\prime \prime} = \bar \varPhi d\bar z$, where $\varPhi$
 takes values in $\mathfrak g$ and $\bar \varPhi$ denotes the 
 complex conjugate.
 Then the conditions \eqref{eq:stronglyharmonic}
 are equivalent with that $(\sym \mu)(\varPhi,\bar{\varPhi})=0$, 
 $[\varPhi,\bar{\varPhi}] =0$ and $\varPhi$ is holomorphic.
 Moreover, since $[\varPhi,\bar{\varPhi}] =0$, we have $[\int_{z_*}^z\varPhi(t) d t,
 \int_{z_*}^{\bar z}\bar{\varPhi}(t) dt] =0$, and the map $\varphi$ is given by 
 \eqref{eq:varPhi}, up to left translation.
\end{proof}
\begin{Remark}
 This type of solutions has been investigated by \cite{Jensen-Liao}
 for complex projective spaces. 
\end{Remark}
 It should be recalled that 
 by Proposition \ref{allconnections}
 the torsion of a map $\varphi:M\to (G,\nt)$ can be computed as 
 $\varphi^{*}T^{{}^{(t)}\mu}(\alpha^{\prime} \wedge \alpha^{\prime \prime}) 
 = t [\alpha^{\prime}\wedge \alpha^{\prime\prime}]$.
 Thus a map $\varphi:M\to (G,\nt),  t\neq 0,$ satisfies the torsion free condition if 
 and only if 
\be
 [\alpha^{\prime}\wedge \alpha^{\prime\prime}] =0.
\ee
 Moreover the bilinear map 
 ${}^{(t)}\mu$ in \eqref{def:threeconnections} is skew-symmetric, that is, the admissibility condition is
 automatically satisfied, and thus $\varphi$ admits the zero-curvature representation.
 We summarize the above discussion as follows:
\begin{Corollary}\label{coro:torsion-freeharmonicity}
 Let $\varphi:M\to (G, \nt)$ be a 
 torsion free $\nt$-harmonic map $(t \neq 0)$ into a Lie group with
 a left-invariant connection $\nt, t \neq 0,$ and 
 $\alpha = \varphi^{*}\theta  = \alpha^{\prime} + 
 \alpha^{\prime \prime}$ the Maurer-Cartan form of $\varphi$.
 Then the loop of connections $d+\alpha_{\lambda}$, defined by
\begin{equation}\label{eq:alpha-familystrong2}
 \alpha_{\lambda}:= \lambda^{-1}\alpha^{\prime} 
 +\lambda \alpha^{\prime\prime},
\end{equation}
 is flat for all $\lambda \in \mathbb S^1$. 
 Conversely assume that  $\mathbb{D}$ is simply-connected and 
 let $\alpha_{\lambda}= \lambda^{-1}\alpha^{\prime}+ \lambda
 \alpha^{\prime \prime}$ be an $\mathbb{S}^1$-family of 
 $\mathfrak{g}$-valued $1$-forms satisfying 
 the zero-curvature condition
\be
 d\alpha_{\lambda}+\frac{1}{2}[\alpha_\lambda\wedge \alpha_\lambda]=0
\ee
 for all $\lambda \in \mathbb{S}^1$.
 Then there exists a $1$-parameter family of maps $\varphi_\lambda:
 \mathbb{D}\times \mathbb{S}^1\to G$
 such that  $\varphi^{*}\theta= \alpha_{\lambda}$. 
 The maps $\varphi_{\lambda}$ are torsion-free $\nt$-harmonic maps 
 for any $t \neq 0$ and $\lambda \in \mathbb S^1$. 
 Moreover such torsion-free $\nt$-harmonic maps $\varphi_{\lambda}$, 
 up to left translation, are given as 
\be
 \varphi_{\lambda}(z, \bar z) = 
 \exp\left(\lambda^{-1} \int_{z_*}^z\varPhi(t) dt +\lambda \int_{z_*}^{\bar z} \bar{\varPhi}(t) dt\right),
\ee
 where $z \in \D$ is conformal coordinates and 
 $\varPhi$ is holomorphic, and its conjugate $\bar \varPhi$ satisfies
 $[\varPhi, \bar \varPhi] =0$.
\end{Corollary}

\subsection{Admissible harmonic maps into Lie groups with left-invariant metrics}
 Clearly, a particularly interesting type of connections consists of 
 the Levi-Civita connections of left-invariant metrics on semi-simple Lie groups.
 Therefore, and more generally, in this section  we consider arbitrary Lie groups 
 together with a left-invariant semi-Riemannian metric.
 We denote a non-degenerate symmetric bilinear form   on the Lie 
 algebra $\mathfrak{g}$ of the  real Lie group $G$  by $\langle\cdot,\cdot\rangle$ 
 and extend it to a left-invariant semi-Riemannian metric 
 $ds^2=\langle \cdot,\cdot\rangle$ on $G$.
 Next we introduce an \textit{anti-commutator product} 
 $\{\cdot,\cdot\}$ on $\mathfrak{g}$ by 
\begin{equation}\label{NR}
 \{X,Y\}:=\nabla_{X}Y+\nabla_{Y}X, \ \  X,Y\in \mathfrak{g}.
\end{equation}
 Clearly $\{X,Y\}=\{Y,X\}$.
 It is easy to see that  $ds^2$ is right invariant if and only if 
 $\{\cdot, \cdot\}$ vanishes.
 The Levi-Civita connection $\nabla$ of $(G,ds^2)$ is given by 
\cite[Proposition 3.18]{CE}, 
\begin{equation}\label{Levi-CivitaRelation}
 \nabla_{X}Y=\frac{1}{2}[X,Y]+\frac{1}{2}\{X, Y\}, \;\;X, Y \in \mathfrak g.
\end{equation}
 Hence $\nabla$ is the left-invariant connection 
 $\nabla = \nabla^{\mu}$  associated with  the bilinear map $\mu$ given by
\be
 \mu(X, Y)=\frac{1}{2}[X,Y] +\frac{1}{2}\{X, Y\}.
\ee
This implies immediately the following Lemma.
\begin{Lemma}
The Levi-Civita connection $\nabla$ of some  left-invariant metric 
 $ds^2$ on $G$  is bi-invariant if and only if $\nabla=\neutral$.
\end{Lemma}
 Thus the Levi-Civita connection $\nabla$ is a left-invariant 
 connection determined by the bilinear map $\mu$ such that 
\begin{equation}\label{eq:skew-symm}
(\sk \mu)(X,Y)=\frac{1}{2}[X,Y],
\;\;
(\sym \mu)(X,Y)=\frac{1}{2}\{X,Y\}.
\end{equation}
 Let us consider a smooth map $\varphi:M\to (G,ds^2)$ 
 from a Riemann surface $M$ into a Lie group $G$ with a 
 left-invariant metric $ds^2$, 
 and let $\nabla$ denote the Levi-Civita connection of $ds^2$.

 From Corollary \ref{coro:affharmonic}, the $\nabla$-harmonicity is characterized 
 by \eqref{0-pluri}. Moreover since 
 the symmetric part of $\mu$ is given in \eqref{eq:skew-symm}, we thus have 
\begin{equation}\label{harm+integ}
  2 \bar \partial \alpha^{\prime } + \{\alpha^{\prime \prime}\wedge 
 \alpha^{\prime}\} + [\alpha^{\prime \prime} \wedge \alpha^{\prime}]=0.
\end{equation}
 In the study of harmonic maps of Riemann surfaces into 
 compact semi-simple Lie groups equipped with a
 bi-invariant Riemannian metric, the zero curvature representation is 
 the starting point of the loop group approach 
 by Uhlenbeck \cite{Uhlenbeck} and Segal \cite{Segal}.
 We recall that $\{\cdot,\cdot\}$
 is the symmetric part of $\mu$ by \eqref{eq:skew-symm}, 
 thus  the admissibility condition \eqref{eq:admissible} is satisfied  if and only if 
 \begin{equation}\label{eq:admMetricLie}
 \{ \alpha^{\prime} \wedge \alpha^{\prime \prime}\} =0.
 \end{equation}

 Therefore, from Proposition \ref{prop:admissibleharmonic} 
 we obtain the following.
\begin{Proposition}\label{ZCRforG}
 Let $\varphi:\mathbb{D}\to (G, ds^2)$ be 
 an admissible harmonic map with respect 
 to the Levi-Civita connection 
 and $\alpha = \varphi^{*}\theta = 
 \alpha^{\prime} + \alpha^{\prime \prime}$ the Maurer-Cartan 
 form. 
 Then the loop of connections $d+\alpha_{\lambda}$ defined by
\begin{equation}
 \alpha_{\lambda}:=
 \frac{1}{2}(1-\lambda^{-1})\alpha^{\prime}+
 \frac{1}{2}(1-\lambda)\alpha^{\prime \prime}
\end{equation}
 is flat for all $\lambda \in \mathbb S^1$. 

 Conversely assume that 
 $\mathbb{D}$ is simply-connected. Let $\alpha_{\lambda}=
 \frac{1}{2}(1-\lambda^{-1})\alpha^{\prime}+
 \frac{1}{2}(1-\lambda)\alpha^{\prime \prime}$ 
 be an $\mathbb{S}^1$-family of $\mathfrak{g}$-valued 
 $1$-forms satisfying \eqref{eq:admMetricLie} and 
 the zero-curvature condition
\be
 d\alpha_{\lambda}+\frac{1}{2}[\alpha_\lambda\wedge \alpha_\lambda]=0
\ee
 for all $\lambda \in \mathbb{S}^1$.
 Then there exists a $1$-parameter family of maps $F_\lambda:
 \mathbb{D}\times \mathbb{S}^1\to G$ such that 
 $F_{\lambda}^{*}\theta=\alpha_{\lambda}$. 
 The map $F_{\lambda}$ is called the extended solution associated with 
 $\alpha_\lambda$ and 
 it is admissible harmonic with respect to the Levi-Civita connection 
 for $\lambda=\pm 1$. 
\end{Proposition}
\begin{Remark}
 The extended solution 
 $F_{\lambda}$ is not harmonic for $\lambda\not=\pm 1,$ but 
 is a solution to the \textit{harmonic map equation with Wess-Zumino term}
 \cite{Hitchin2, Khemar, Tyrin}. 
 In case, the metric is bi-invariant, then 
 $\varphi_{\lambda}:=F_{-\lambda}F_{\lambda}^{-1}$ is harmonic 
 for all $\lambda\in \mathbb{S}^1$, see \cite{DE}.  
\end{Remark}
\subsection{The admissibility condition}
\label{subsc:admissibility}
 In the previous subsections we have characterized admissible affine harmonic maps 
 by flat connections. Thus the admissibility condition \eqref{eq:admissible}
 is important for the loop group method. So the following 
 two questions naturally arise:
 \begin{enumerate}
\item Does there exist a \textit{non} skew-symmetric $\mu$ such that 
 \textit{every} $\nm$-harmonic map satisfies the admissibility condition 
 \eqref{eq:admissible}?
\item Does there exist any non-skew-symmetric $\mu$ such that 
 the connection $\nm$ admits an admissible $\nm$-harmonic map?
 \end{enumerate}
 The answer to the first question is as follows:
\begin{Proposition}\label{pro:admissibleandmu}
 Let $G$ be a Lie group and $\nm$ a left-invariant connection
 determined by a bilinear map $\mu$. 
 Assume that every $\nm$-harmonic map into $(G, \nm)$
 satisfies the admissibility 
 condition \eqref{eq:admissible}. 
 Then $\mu$ is skew-symmetric.
\end{Proposition}
\begin{proof}
 Let $z = x+ \sqrt{-1}y$
 be a conformal coordinate on $\D\subset \C$, and set 
 $\alpha^{\prime} = A(x) dz$
 and $\alpha^{\prime\prime} = A(x) d\bar z$, where $A (x)$ is 
 an arbitrary function of $x$ which takes values in $\mathfrak g$. Then 
 it is clear that the Maurer-Cartan equation \eqref{eq:MCforalpha}
 is automatically satisfied. Moreover, the 
 $\nm$-harmonicity equation \eqref{eq:mu-harmonic} is equivalent with 
 the ordinary differential equation for $A(x)$:
 \begin{equation}\label{eq:odeforA}
 \frac{d}{d x} A(x) + 2 \mu (A(x), A(x))=0.
 \end{equation}
 Consider $A(x)$ 
 with (arbitrarily chosen) initial condition $A_0 = A(0) \in \mathfrak g$.
 Therefore there exists a $\nm$-harmonic map satisfying $\alpha = \varphi^* \theta 
 = A(x) dz + A(x) d \bar z$. 
 
 On the other hand, by the admissibility condition \eqref{eq:admissible} we have 
 $\mu (A(x), A(x))=0$.
 In particular $\mu (A(0), A(0))=0$. Since the initial condition $A_0 =A(0)$
 for \eqref{eq:odeforA} was chosen arbitrary, we conclude that 
\be
 \mu (A_0, A_0)=0
\ee
 for any $A_0 \in \mathfrak g$. Thus $\mu$ is skew-symmetric.
\end{proof}
 To answer the second question, we introduce the 
 notion 
of a \textit{ primitive map} into $k$-symmetric Lie groups ($k>2$). 
 Let $G$ be a connected Lie group with a Lie group automorphism $\tau$ of order $k>2$ 
 such that the unit element is the only element of $G$ fixed by $\tau$.
 The pair $(G,\tau)$ is referred to as a $k$-\textit{symmetric Lie group}. 
 A $k$-symmetric Lie group $(G,{}^{(-1)}\nabla)$ equipped 
 with the canonical connection 
 is said to be \textit{affine $k$-symmetric} if $\tau$ is an affine transformation 
 with respect to the canonical connection. Note that affine $k$-symmetric Lie groups are 
 affine $k$-symmetric spaces in the sense of Kowalski \cite{Kowalski}. Note that 
 affine $k$-symmetric Lie groups are solvable \cite[Proposition V.9]{Kowalski}.

 If we equip $G$ with a left invariant semi-Riemannian metric $ds^2$  
 such that $\tau$ is an isometry with respect to this metric, 
 then the resulting semi-Riemannian homogeneous space 
 $(G/\{\mathrm{Id}\},\tau,ds^2)$ is a semi-Riemannian $k$-symmetric space in the sense of 
 \cite{Kowalski}.

 The eigenvalues of $d\tau$ on $\mathfrak{g}$ are 
 contained in the set
 $\{\omega^{j}\ \vert \ j\in \mathbb{Z}/k\mathbb{Z}\}$, where 
 $\omega=e^{2\pi{i}/k}$ is the primitive $k$-th root of unity.  
 We have an eigenspace decomposition
 of the complexification
 $\mathfrak{g}^\mathbb{C}$ of $\mathfrak{g}$:
\be
\mathfrak{g}^\mathbb{C}=
\sum_{j \in \mathbb{Z}/\mathbb{Z}_k}\mathfrak{g}^\mathbb{C}_{j},
\ee
 where ${\mathfrak g}_{j}^{\mathbb C}$ is the
 eigenspace of $d\tau$ 
 corresponding to the eigenvalue $\omega^j$.
 Clearly, ${\mathfrak g}_{0}^\mathbb{C}=\{\Vec{0}\}$. 
 Moreover we have 
\be
[{\mathfrak g}^\mathbb{C}_{j},\ {\mathfrak g}^\mathbb{C}_{\ell}]=
{\mathfrak g}^\mathbb{C}_{j+\ell} 
\ \mod k.
\ee
 Note that ${\mathfrak g}^\mathbb{C}_{1}\cap
 {\mathfrak g}^\mathbb{C}_{-1}=\{{\mathbf 0}\}$
 since $k>2$.
\begin{Definition}
 Let $\varphi:M\to (G,\tau)$ be a smooth map from a Riemann surface
 into a $k$-symmetric Lie group. Then $\varphi$ is said to be 
 a \textit{primitive map} if
\be
\alpha^{\prime}(T^{(1,0)}M)\subset \mathfrak{g}^{\mathbb C}_{-1}.
\ee
 Here $\alpha^{\prime}$ is the $(1,0)$-part of $\alpha=\varphi^{*}\theta$.
\end{Definition}
 Note that the notion of  primitivity does not depend
 on any choice of left-invariant 
 affine connection. If we express the Maurer-Cartan form 
 $\alpha=\varphi^{-1}d\varphi$ as $\alpha=\alpha^{\prime}+\alpha^{\prime\prime}$ with 
 $\alpha^{\prime}=\varPhi\>dz$, then $\varphi$ is primitive 
 if and only if $\varPhi\in \mathfrak{g}_{-1}^{\C}$. 

 We obtain the following.
\begin{Proposition}
 Let $G$ be a $k$-symmetric Lie group with automorphism $\tau$ of order $k>2$. 
 If a smooth map $\varphi:M\to G$ of a Riemann surface
 $M$ is primitive with respect to $\tau$ then $\varphi$ has the form 
\begin{equation}\label{eq:varphiform}
 \varphi(z,\bar{z})=a \exp \left(\int_{z_*}^z \varPhi(t) dt
 +\int_{z_*}^{\bar z}\bar{\varPhi}(t) dt\right),
\end{equation}
 where $a\in G,$ and $\varPhi$ is holomorphic, takes 
 values in $\mathfrak g^{\mathbb C}$,
 and satisfies $[\varPhi,\bar \varPhi]=0$.
\end{Proposition}
\begin{proof}
 Let $\alpha = \varPhi dz + \bar \varPhi d\bar z$ be the Maurer-Cartan form of 
 the primitive map $\varphi$.
 By definition, $\varPhi\in \mathfrak{g}^{\mathbb C}_{-1}$, so we have 
 $[\varPhi,\bar \varPhi]\in \mathfrak{g}^{\mathbb C}_{0}=\{\Vec{0}\}$.
 Next $\alpha$ satisfies the Maurer-Cartan equation
\be
\bar{\partial}\alpha^{\prime}+\partial \alpha^{\prime\prime}
+[\alpha^{\prime}\wedge \alpha^{\prime\prime}]=0.
\ee
Since $\varphi$ is a primitive map, we obtain 
\be
\bar{\partial}\alpha^{\prime}+\partial \alpha^{\prime\prime}=0.
\ee
 Split this equation according to the eigenspace decomposition 
 of $\mathfrak{g}^{\mathbb{C}}$ with respect to $\tau$, we get
 $\bar{\partial}\alpha^{\prime}=\partial \alpha^{\prime \prime}=0$.
 Hence by Theorem \ref{thm:strongharmonic}, $\varphi$ has the form 
 \eqref{eq:varphiform}. 
\end{proof}
 We now answer the second question posed at the beginning of this subsection as follows:
\begin{Proposition}
 Let $(G,\tau,ds^2)$ be Lie group equipped with the 
 structure of semi-Riemannian $k$-symmetric space $(k>2)$. Assume that 
 the metric $ds^2$ is only left invariant. Then the 
 Levi-Civita connection of $ds^2$ is of the form $\nabla^\mu$ with non skew-symmetric $\mu$. 
 The admissibility condition $\{\alpha^{\prime}\wedge \alpha^{\prime\prime}\}=0$ 
 is satisfied for every primitive map from a Riemann surface $M$ into $(G,\tau,ds^2)$.
\end{Proposition}
 For example, the space $\mathrm{Sol}_3$, which will be defined 
 in Section \ref{subsc:Sol3}, 
 has the structure of a Riemannian $4$-symmetric space \cite{Kowalski}, 
 see Section \ref{subsc:primitive}. 
 Thus every primitive map of a Riemann surface into $\mathrm{Sol}_3$ 
 is harmonic with respect to 
 its standard left invariant metric $ds^2$ and satisfies the admissibility condition. 
\section{Basic Examples}\label{sc:BasicExamples}
 In the discussions of the previous sections we encountered several 
 interesting basic types of 
 classes of affine harmonic maps into  Lie groups $G$. 
\begin{enumerate}
 \item The class of $\neutral$-harmonic maps into $G$.

\item The class of admissible harmonic maps with respect to a (fixed)
 left-invariant metric.

\item The class of torsion-free $\nt$-harmonic maps for any  
      $t \in \mathbb R \setminus \{0\}$.
\end{enumerate}
 All harmonic maps of these classes admit zero-curvature representations.

 In the first case, the Lie group $G$ is regarded 
 as an \textit{affine symmetric space} 
 $(G\times{G}/G,\neutral)$. This case contains the Uhlenbeck-Segal theory, 
 when $G$ is compact and semi-simple. In fact, if $G$ is compact and 
 semi-simple, 
 $\neutral$ coincides with the Levi-Civita connection of 
 the (actually bi-invariant) Killing metric. 

 If  $G$ is not semi-simple, $G$ may not have any  bi-invariant 
 semi-Riemannian metric (like in the case of the Heisenberg group, see below).
 Even if we equip $G$ with a left-invariant metric, $\neutral$-harmonic 
 maps are 
 not necessarily harmonic with respect to the Levi-Civita connection, see Section \ref{subsc:Nil}.
 Only in the bi-invariant case these coincide.

 In the second case, unfortunately, we 
 only know very simple examples, like 
 the map given by \eqref{eq:varPhi} 
 and the primitive harmonic maps discussed in 
 Section \ref{subsc:admissibility}.
 
 In the third case, torsion-free $\nt$-harmonic maps $(t \neq 0)$, 
 we need to demand the additional condition ``torsion-free'', 
 that is, $[\alpha^\prime\wedge \alpha^{\prime\prime}]=0$, 
 otherwise we do not  obtain a zero curvature representation.
 Torsion-free $\nt$-harmonic maps have the advantage that 
 they are harmonic with respect to any left-invariant metric.
 However, unfortunately, the torsion-free condition is a very 
 strong restriction on $\nt$-harmonic maps. In fact such maps 
 have been classified in Theorem \ref{thm:strongharmonic}.

 Therefore the first case is the only candidate 
 for a generalization of the Uhlenbeck-Segal theory.
 Among all Lie groups, two particularly interesting cases occur, 
 the semi-simple Lie groups and the solvable Lie groups. 
 Since the semi-simple case has been studied intensively already, 
 in this section we exhibit some interesting examples 
 of solvable Lie groups. 
 We will start, however, with a completely general example.

\subsection{An equivariant map}
 Let $G$ be any Lie group and $X,Y$ any pair of vectors in $\mathfrak{g}$.
 We define a map $\varphi:\mathbb{C}\to G$ by 
\be
\varphi(x,y)=\exp (xX)\exp(yY).
\ee
 Then the Maurer-Cartan form $\alpha$ of $\varphi$ is given by
 $\alpha =\ad (\exp(-yY))X dx + Ydy  = \alpha' + \alpha'',$ where
 $\alpha' = \frac{1}{2}( \ad (\exp(-yY))X -\sqrt{-1} Y)dz$ and $z = x + \sqrt{-1} y$ 
 is a conformal coordinate.
 From this we get 
 $\partial\alpha^{\prime \prime} 
 - \bar \partial \alpha^{\prime}=0$. 
 Hence, by  (\ref{eq:mu-harmonic})  we infer that $\varphi$ is $\neutral$-harmonic. 
 When $G$ admits a bi-invariant semi-Riemannian 
 metric $ds^2=\langle\cdot, \cdot\rangle$, then 
 $\varphi$ is harmonic and the induced metric is computed as
\be
 \langle X,X\rangle dx^2+2\langle X,Y \rangle dxdy+
 \langle Y,Y\rangle dy^2.
\ee
 If in addition we take $X$, $Y$ so that 
 $\langle X,X\rangle=\langle Y,Y\rangle>0$
 and $\langle X,Y\rangle=0$, then $\varphi$ 
 is a conformally harmonic immersion.
 For $t\not=0$, 
 $\varphi$ is torsion free ${}^{(t)}\nabla$-harmonic 
 if and only if $[X,Y]=0$, see
 Theorem \ref{thm:strongharmonic}.
\begin{Remark}
 When $[X, Y]=0$, $\varphi(x, y) = \exp (x X) \exp(y Y)$
 is called a {\it vacuum solution}, \cite{BP}.
\end{Remark}

\subsection{The Heisenberg group}\label{subsc:Nil}
 We consider the $(2n+1)$-dimensional Heisenberg group 
 $\mathrm{Nil}_{2n+1}$ in $\mathbb{R}^{2n+1}$ with multiplication
\begin{align*}
(x^1,\cdots,x^{2n},x^{2n+1})&\cdot
(\tilde{x}^1,\cdots,\tilde{x}^{2n},\tilde{x}^{2n+1}) \\
&=\left(x^1+\tilde{x}^1,\cdots,x^{2n}+\tilde{x}^{2n},
x^{2n+1}+\tilde{x}^{2n+1}+\frac{1}{2}\sum_{i=1}^{n}
(x^{i}\tilde{x}^{n+i}-\tilde{x}^{i}x^{n+i})\right).
\end{align*}
 The natural left-invariant metric is 
\be
ds^2=\sum_{i=1}^{2n}(dx^i)^2+
 \left\{dx^{2n+1}+\sum_{i=1}^{n}(x^{n+i}dx^i-x^idx^{n+i})
 \right\}^2.
\ee
 It is known that $\mathrm{Nil}_{2n+1}$
 has no bi-invariant metric \cite{MR}.
 For a smooth map $\varphi=(\varphi^1, \dots, \varphi^{2n+1}):
 \mathbb C \to \mathrm{Nil}_{2n+1}$, the $\neutral$-harmonicity 
 equation is 
\be
 \varphi^j_{z \bar z} =0,\;\;(1 \leqq j \leqq 2 n), \;\;
 \left\{\varphi^{2n+1}
 +\frac{1}{2}\left(\sum_{i=1}^n\varphi^i\varphi^{n+i}\right)
 \right\}_{z\bar{z}}
 -\frac{1}{2}\sum_{i=1}^n\left(\varphi^i_z\varphi^{n+i}_{\bar z}
 +\varphi^i_{\bar z}\varphi^{n+i}_z\right)=0.
\ee
 This system is equivalent to 
 $\varphi^1_{z\bar{z}}=\varphi^2_{z\bar{z}}
 =\cdots = \varphi^{2n +1}_{z\bar{z}}=0$.
 Thus every $\neutral$-harmonic map can be 
 represented in the form
\be
\varphi= (\varphi^1,\dots \varphi^{2n+1}) =
 (2 \Re f^1, \dots,2 \Re f^{2n+1}),
\ee
 where $f^1, \dots, f^{2n +1}$ are holomorphic functions.

 Conversely, any such $\varphi$ is a $\neutral$-harmonic map.
 For example, we consider $n=1$ and 
 take $ f^1(z)=\frac{1}{2}z, f^{2}(z)=-\frac{\sqrt{-1}}{2}z$ and 
 $f^{3}(z)=-\frac{\sqrt{-1}}{8}z^2$, where $z = x + \sqrt{-1} y$.
 Then the resulting map
 $\varphi=(x,y,\tfrac{1}{2}xy)$ 
 is a $\neutral$-harmonic map and this map gives 
 a hyperbolic paraboloid $x^3=\tfrac{1}{2} x^1x^2$ 
 which is a standard example 
 of minimal surfaces in $\mathrm{Nil}_3$ with respect to the 
 canonical left-invariant metric, see \cite{DIK2}.
 Note that the hyperbolic paraboloid is represented as
\be
 \varphi(x, y) = 
\exp \left( x 
 \begin{pmatrix} 
 0 & 1 & 0\\  0 & 0 & 0\\  0 & 0 &0 \\ 
 \end{pmatrix}
\right)
\exp \left(y
 \begin{pmatrix} 
 0 & 0 & 0\\  0 & 0 & 1\\  0 & 0 &0 \\ 
\end{pmatrix}
\right).
\ee
 The canonical connection $\can$ coincides with 
 the \textit{Tanaka-Webster connection} in CR-geometry. 
 In \cite{CIL} the following result was obtained.
\begin{Proposition}
 Let $(M,g)$ be a Riemannian $2$-manifold and $\varphi:M\to 
 (\mathrm{Nil}_3,ds^2)$  an isometric immersion. 
 Then $\varphi$ is $\can$-harmonic if and only if $\varphi$ is 
 locally congruent to the vertical plane. In particular 
 $\varphi$ is harmonic with 
 respect to the canonical left-invariant metric. 
\end{Proposition}

For instance, the vertical plane $x_2=0$ in $\mathrm{Nil}_3$ 
is represented as 
\be
\varphi(z,\bar{z})=\exp(z\varPhi+\bar{z}\bar{\varPhi}),
\ \ 
\varPhi=\frac{1}{2}
\left(
\begin{array}{ccc}
0 & 1 & -\sqrt{-1}\\
0 & 0 & 0\\
0 & 0 & 0
\end{array}
\right).
\ee
 One can easily check that $[\varPhi,\bar {\varPhi}]=0$.
 Hence the vertical plane is a torsion-free $\can$-harmonic surface.
\subsection{An interesting family of solvable Lie groups}
\label{subsc:Sol3}
 Let us consider the following $2$-parameter 
 family, $(\mu_1,\mu_2) \in \mathbb{R}^2 \setminus (0,0)$, of $3$-dimensional
 Lie groups:
\begin{equation}\label{eq:solvable}
G(\mu_1,\mu_2)=\left\{
(x^1,x^2,x^3)=\left(
\begin{array}{cccc}
 e^{\mu_1\>x^3} & 0 & x^1\\
 0 & e^{\mu_2\>x^3} & x^2\\
 0 & 0 &1 
\end{array}
\right)
\right\}
\subset 
\mathrm{GL}_{3}\mathbb{R}.
\end{equation}
 The Lie group $G(\mu_1,\mu_2)$ is solvable 
 and non-unimodular if $\mu_1+\mu_2\not=0$. 
 Note that $G(\mu_1, \mu_2)$ is originally defined as 
 $\mathbb R^3 (x^1, x^2, x^3)$ with multiplication law: 
\be
 (x^1, x^2, x^3) \cdot (\hat x^{1}, \hat x^{2}, \hat x^{3}) =
 ( x^1 + e^{\mu_1 x^3} \hat x^{1}, x^2 + e^{\mu_2 x^3} \hat x^{2}, x^3
 + \hat x^{3})
\ee
 for any $(\mu_1, \mu_2) \in \mathbb R^2$. In case $(\mu_1, \mu_2) \neq (0, 0)$, 
 $G(\mu_1, \mu_2)$ is realized as the matrix Lie group given by \eqref{eq:solvable}.
 We note that in case $\mu_1 = \mu_2=0$, $G(\mu_1, \mu_2)$ 
 is just the abelian group $\mathbb R^3$ 
 and one can not realize this group as a matrix group in 
 \eqref{eq:solvable}.
 The $\neutral$-harmonic map equation for 
 $\varphi=(\varphi^1,\varphi^2,\varphi^3)$ is 
\begin{equation}\label{sol-0-harm}
\varphi^{k}_{z\bar{z}}-\frac{1}{2}\mu_{k}
(\varphi^k_z\varphi^3_{\bar z}+\varphi^k_{\bar z}\varphi^3_z)=0, \>(k=1,2),
\ \ \varphi^3_{z\bar{z}}=0.
\end{equation}
  We equip $G(\mu_1,\mu_2)$ with the left-invariant metric
\be
 ds^2 =e^{-2\mu_1\>x^3}(dx^1)^2+e^{-2\mu_2\>x^3}(dx^2)^2+(dx^3)^2.
\ee

 The resulting family of Riemannian homogeneous spaces includes
 Euclidean $3$-space $\mathbb{E}^3=G(0,0)$, hyperbolic $3$-space 
 $\mathbb{H}^3(-c^2)=G(c,c)$ of curvature 
 $-c^2<0$, the model space $\mathrm{Sol}_3=G(1,-1)$ 
 of \textit{solvgeometry}  
 in the sense of Thurston and the Riemannian product 
 $\mathbb{H}^{2}(-c^2)\times \mathbb{E}^1
 =G(0,c)$. 
 We also note that for a group $G(\mu_1,\mu_2)$  the metric introduced above is  bi-invariant if and 
 only if $\mu_1 = \mu_2 =0$.
 In general, the harmonic map equation with respect to this metric is
\begin{equation}\label{sol-harm}
\varphi^{k}_{z\bar{z}}-\mu_{k}
(\varphi^k_z\varphi^3_{\bar z}+\varphi^k_{\bar z}\varphi^3_z)=0, \;(k=1,2), \;
\ \ \varphi^3_{z\bar z}+\sum_{k=1}^{2}\mu_{k}
e^{-2\mu_k\>\varphi^3}\varphi^{k}_{z}\varphi^{k}_{\bar z}=0.
\end{equation}
\subsection{Primitive maps into Sol$_3$}\label{subsc:primitive}
 Let $G = \mathrm{Sol}_3$ and 
 $\mathfrak g = \mathrm{Lie} (\mathrm{Sol}_3)$.
 Then 
 $(G, \tau, ds^2)$ is a Riemannian $4$-symmetric space 
 with automorphism $\tau$ defined by 
 $\tau (x^1, x^2, x^3) = (-x^2, x^1, -x^3)$.
 Since the primitive $4$-th root
 of unity is $\omega = \sqrt{-1}$,
 the eigenvalues of $d \tau$ are
 $\pm 1$ and $\pm{\sqrt{-1}}$. 
 One can determine the eigenspaces
 of ${\mathfrak g}^\mathbb{C}$
 corresponding to $\omega^j,  (j = 0, 1, 2, -1)$, explicitly:
\be
{\mathfrak g}^\mathbb{C}_{0}=\{{\mathbf 0}\},\ 
{\mathfrak g}^\mathbb{C}_{1}=
{\mathbb C}
\left[
\begin{array}{c}
1 \\ -\sqrt{-1} \\ 0
\end{array}
\right],
\ \ 
{\mathfrak g}^\mathbb{C}_{2}
 =
\left[
\begin{array}{c}
0 \\ 0 \\ 1
\end{array}
\right]
,\ \
{\mathfrak g}^\mathbb{C}_{-1}=
{\mathbb C}
\left[
\begin{array}{c}
1 \\ \sqrt{-1} \\ 0
\end{array}
\right].
\ee
Direct computations show that
$\varphi=(\varphi^1,\varphi^2,\varphi^3)
:{\mathbb D}\rightarrow \mathrm{Sol}_3$ is primitive if and only if
\be
\varphi^{3}_z=0,\  \ 
\sqrt{-1}e^{-\varphi^3}\varphi^{1}_z=e^{\varphi^3}\>\varphi^{2}_z.
\ee
The first equation implies that $x^3$ is constant.
The second equation is rewritten as
\be
\frac{\partial}{\partial z}
\left(
\varphi^1 - \sqrt{-1}e^{2\varphi^3}\varphi^2
\right)=0.
\ee
 Thus we obtain the following.
\begin{Proposition}
 Let $\varphi:\mathbb{D}\rightarrow G$
 be a smooth map into $\mathrm{Sol}_3$. Then $\varphi$ is a primitive map if and only if
 $\varphi^{3}=\textrm{constant}$ and $\varphi^1$ and ${e}^{2\varphi^3} \varphi^2$ 
 are conjugate harmonic functions.
\end{Proposition}
%

\subsection{Generalization of Sol$_3$}
 The following space is regarded as a higher dimensional 
 generalization of $\mathrm{Sol}_3$.
 Let us consider the solvable Lie group
\be
 G_n=\left\{(w, u^1, \dots, u^n, v^1, \dots, v^n)=
\left(
\begin{array}{ccccc}
e^{t} & 0 & \cdots &0 & w\\
0 & e^{v^1} & \cdots & 0 & u^1\\
\vdots & \vdots & \ddots & \vdots & \vdots\\
0 & 0 &\cdots &  e^{v^n} & u^n\\
0 & 0 &\cdots &  0 &1
\end{array}
\right)
\right\}
\subset 
\mathrm{SL}_{n+2}\mathbb{R}, 
\ee
 where $t=-(v^1+v^2+\cdots+v^n)$. 
 For $n=1$, setting $(x_1,x_2,x_3):=(w,u^1,-v^1)$,
 we obtain $G_1=\mathrm{Sol}_3$.
 The $\neutral$-harmonic map equation for $\varphi:\mathbb{C}\to G_n$ is 
\be
 v^{k}_{z\bar z}=0, \;
 u^{k}_{z\bar z}=\frac{1}{2}(u^k_zv^k_{\bar z}+u^k_{\bar z}v^k_{z}),
 \;(k =1, \cdots n),\;\;
 w_{z\bar z}=-\frac{1}{2}
\sum_{k=1}^{n}
 \left(
 w_{\bar z} v^k_{z}
 +w_{z} v^{k}_{\bar z}
 \right).
\ee
 This space equipped with a left-invariant metric
\be
 ds^2 = \sum_{k=1}^{n}e^{-2v_k}(du^k)^2+\sum _{i,j=1}^{n}dv^idv^j+e^{-2t}(dw)^2
\ee
 is a $(2n+1)$-dimensional Riemannian $(2n+2)$-symmetric 
 space \cite{Bozek, Kowalski}. 

\subsection{The Euclidean motion group}
 The motion group of the Euclidean plane $\mathbb{E}^2$ is 
\be
\widetilde{\mathrm{SE}_2}=
\left\{ (x^1, x^2, x^3) =
\left(
\begin{array}{ccc}
\cos x^3& -\sin x^3& x^1\\
\sin x^3& \cos x^3& x^2\\
0 & 0 & 1 
\end{array}
\right)
\right\}
\subset
\mathrm{SL}_{3}\mathbb{R}. 
\ee
 The $\neutral$-harmonic map equation for 
 $\varphi=(\varphi^1,\varphi^2,\varphi^3):\mathbb{C}\to \mathrm{SE}_2$ 
 is the system: 
\begin{equation}\label{eq:SE2}
 \varphi^1_{z\bar z}+\frac{1}{2}(\varphi^2_{z}\varphi^3_{\bar z}+\varphi^2_{\bar z}\varphi^3_z)=0,\;
\varphi^2_{z\bar z}-\frac{1}{2}(\varphi^1_{z}\varphi^3_{\bar z}+\varphi^1_{\bar z}\varphi^3_z)=0,\;
 \varphi^3_{z\bar z}=0.
\end{equation}
 The standard left-invariant metric of $\mathrm{SE}_2$ is the flat one 
 $ds^2 = (dx^1)^2+(dx^2)^2+(dx^3)^2$. With respect to this flat metric, 
 the harmonic map equation for $\varphi$ is the system:
\be
 \varphi^1_{z\bar z}=\varphi^2_{z\bar z}=\varphi^3_{z\bar z}=0.
\ee
 It is known that $\mathrm{SE}_2$ 
 has no bi-invariant semi-Riemannian metric.
 \begin{Remark}
 The $3$-dimensional Lie groups with  left-invariant Riemannian
 metric have been classified by Milnor \cite{Milnor}. Corresponding 
 results for left-invariant Lorentzian metrics are given in \cite{CP}. 
 Combining these two papers, we obtain the complete list of $3$-dimensional 
 Lie groups with bi-invariant semi-Riemannian metrics.
 \end{Remark}
 
\section{Generalized Weierstrass type representation} 
\label{sc:DPW}
In this section we give the generalized Weierstrass type representation
 for neutral harmonic maps from a Riemann surface into any Lie group $G$.
 We assume that the Lie group $G$
 is a connected real analytic Lie group which 
 admits a faithful finite dimensional representation. 
 Moreover, where it is convenient, we will assume without loss of generality by  
 \cite{Hochschild, Hochschild-Comp} that $G$ is embedded into its linear 
 complexification $G^\C$ and that $G^\C$ is simply-connected.

\subsection{Reductive decompositions of $G$}
 Let $G$ be a real Lie group as before. 
 Then the product Lie group 
 $G\times{G}$ acts transitively on $G$ by 
\be
 (a,b)\cdot g=agb^{-1}.
\ee
 The isotropy subgroup of $G\times{G}$ at the unit element is the 
 diagonal subgroup 
\be
 \triangle=\{(a,a)\in G\times G\ \vert \ a\in G \}.
\ee
 The homogeneous space $ G \cong G\times{G}/\triangle$ is reductive. 
 In fact there are three standard reductive decompositions 
 of the Lie algebra $\mathfrak{g}\oplus \mathfrak{g}$ of $G \times G$:
\begin{align*}
\mathfrak{g}\oplus\mathfrak{g}&=\mathfrak{d}\oplus \mathfrak{p}_{+}, \quad
\mathfrak{p}_{+}=\{(0,X) \ \vert \ X \in \mathfrak{g}\},
\\
\mathfrak{g}\oplus\mathfrak{g}&=\mathfrak{d}\oplus \mathfrak{p}_{-}, \quad
\mathfrak{p}_{-}=\{(X,0) \ \vert \ X \in \mathfrak{g}\},
\\
\mathfrak{g}\oplus\mathfrak{g}&=\mathfrak{d}\oplus \mathfrak{p}_{0},\quad
\mathfrak{p}_{0}=\{(X,-X) \ \vert \ X \in \mathfrak{g}\},
\end{align*}
 where $\mathfrak{d}$ is the Lie algebra of $\triangle$. 
 The canonical connections of $G=G\times{G}/\triangle$ with respect to 
 these reductive decompositions, taking $\mathfrak{p}_*, * \in \{ +,-,0\},$ are $\anti$, 
 $\can$ and $\neutral$, respectively, see \cite[pp.198--199]{KN2}. 
 Among these three reductive decompositions, only
 $\mathfrak{g}\oplus\mathfrak{g}=\mathfrak{d}\oplus \mathfrak{p}_{0}$
 defines a symmetric pair $(\mathfrak{g}\oplus\mathfrak{g},\mathfrak{d})$.
 The corresponding involution $\sigma$ of $G\times{G}$ is 
 given by 
\begin{equation}\label{sigma}
 \sigma(a,b)=(b,a).
\end{equation}
 Moreover, the projection $\pi : G \times G \to G$ is given 
 by $\pi (g, h) =g h^{-1}$.
 Since symmetric spaces are particularly convenient for the loop group 
 method discussed below, we will  consider from now on exclusively the third reductive 
 decomposition.
\begin{Remark}
 From another point of view, the neutral connection $\neutral$  
 is necessary for the generalized Weierstrass type representation:
 Our goal is to consider connections $\nabla^\mu$ on $G$
 for which every $\nabla^\mu$-harmonic map is admissible. 
 We have shown in Proposition \ref{pro:admissibleandmu} 
 that each such connection comes from 
 a skew-symmetric $\mu$. But in $(3)$ of Lemma \ref{lem:skewconnection} 
  we have seen that  a map is  $\nabla^\mu$ 
 harmonic if and only if it is ${}^\dag\nm$-harmonic. Finally, 
 from $(2)$ of Lemma \ref{lem:skewconnection} 
 we know that for skew-symmetric $\mu$ we have ${}^\dag\nm_X Y=
 \neutral_X Y := \frac{1}{2}[X, Y]$.
\end{Remark}
\subsection{Flat connections}
 As we have seen in the preceding sections, a map 
 $\varphi : \mathbb D  \to G$ from a simply-connected 
 domain  $\mathbb D \subset \mathbb C$ is $\neutral$-harmonic 
 if and only if $d + \alpha_{\lambda}$ is a family of 
 flat connections for all $\lambda \in \mathbb S^1$, 
 where 
\begin{equation}\label{eq:alphalambda}
 \alpha =\varphi^{*}\theta  = \alpha^{\prime} + \alpha^{\prime \prime}\;\; \mbox{and} \;\;
 \alpha_{\lambda} = \frac{1}{2}(1- \lambda^{-1}) \alpha^{\prime}+ 
 \frac{1}{2}(1- \lambda) \alpha^{\prime \prime}.
\end{equation}
 First we note that the map $\varphi$ can be (globally) lifted to the frame 
\be
 \mathcal F : \mathbb D \to \G:=G \times G, 
 \;\; p \mapsto (\id, \varphi(p)),
\ee
 where $\id$ denotes the unit of $G$.
 As a consequence 
\begin{equation}\label{eq:doubleMC}
 \mathcal A =\mathcal F^{*}\theta_{\mathcal{G}}  = 
 (0,  \varphi^{*}\theta)  = ( 0, \alpha).
\end{equation}
 Here $\theta_{\mathcal G}$ is the left Maurer-Cartan 
 form of $\mathcal{G}=G\times{G}$.
 Next we decompose $\mathcal A$ with respect to the eigenspaces
 of $d \sigma$. It is easy to verify that the fixed point algebra $\mathcal K$ 
 of $d \sigma$ and the eigenspace $\mathcal P$ 
 for the eigenvalue $-1$ are 
 $\K = \{(X, X)\;|\; X \in \mathfrak g\}$ and 
 $\P = \{(Y, -Y)\;|\; Y \in \mathfrak g\}$.
 As a consequence 
\be
 \A = \A_{\K} +  \A_{\P},
\ee
 where $\A_{\K}= \tfrac{1}{2}(\alpha, \alpha)$
 and $\A_{\P} = \tfrac{1}{2}(-\alpha, \alpha)$.
 After decomposing further we obtain
\be
 \A_{\P} 
 = \A_{\P}^{\prime} + \A_{\P}^{\prime \prime},
\ee
 where $\A_{\P}^{\prime}$ is a $(1, 0)$-form and 
 $\A_{\P}^{\prime \prime}$ is a $(0, 1)$-form.
 We now introduce $\lambda$ as usual:
\begin{equation}\label{eq:Alambda}
 \A_{\lambda} = 
  \lambda^{-1} \A_{\P}^{\prime} 
 + \A_{\K} 
 + \lambda \A_{\P}^{\prime \prime}, \;\;\lambda \in \mathbb S^1.
\end{equation}
 Then a straightforward computation shows 
\begin{equation}\label{eq:Alambda2}
 \A_{\lambda} = 
 \left(
  \frac{1}{2}(1-\l^{-1}) \alpha^{\prime} 
  +\frac{1}{2}(1-\l) \alpha^{\prime \prime}, 
  \frac{1}{2}(1+\l^{-1}) \alpha^{\prime} 
  +\frac{1}{2}(1+\l) \alpha^{\prime \prime}
  \right)
  =(\alpha_{\l},\alpha_{-\l}).
\end{equation}
 Clearly, $\A_{\lambda}$ is integrable, that is, 
 $d \A_{\lambda} + \tfrac{1}{2}[\A_{\lambda} \wedge \A_{\lambda}] 
 =\mathbf 0$, if and only if $\alpha_{\lambda}$ is 
 integrable, that is, $d \alpha_{\lambda} + \tfrac{1}{2}[\alpha_{\lambda} 
 \wedge \alpha_{\lambda}] =0$. 
 Hence there exits a map 
 $\mathcal{F}_{\lambda}:\mathbb{D}\times\mathbb{S}^1\to \mathcal{G}$ 
 satisfying $\mathcal{F}_{\lambda}^{*}
 \theta_\mathcal{G}=\mathcal{A}_\lambda$. The map 
 $\mathcal{F}_\lambda$ is called an \textit{extended frame} of $\varphi$. 
 The extended frame $\mathcal{F}_{\lambda}$ has the form 
\be
 \mathcal{F}_{\lambda}=(F_{\l},F_{-\l}),
\ee
 where $F_{\lambda}$ is an extended solution.
 The following theorem is an immediate consequence 
 of Theorem \ref{thm:familyofconnections}.
\begin{Theorem}\label{thm:Alambda}
 Let $M$ be a Riemann surface, $G$
 a Lie group and $\neutral$ its neutral connection.
 Moreover, let $\varphi : M \to 
 (G, \neutral)$ be a smooth map and $\alpha_{\lambda}$ and 
 $\mathcal A_{\lambda}$ the $1$-forms defined in \eqref{eq:alphalambda} 
 and \eqref{eq:Alambda}.
 Then the following statements are equivalent$:$
\begin{enumerate}
 \item $\varphi$ is $\neutral$-harmonic. 
 \item $d + \alpha_{\lambda}$ is a family of flat connections 
 for all $\lambda \in \mathbb S^1$.
 \item $d + \A_{\lambda}$ is a family of flat connections 
 for all $\lambda \in \mathbb S^1$.
\end{enumerate}
\end{Theorem}
\begin{Remark}
 Every semi-Riemannian symmetric space $G/K$ with 
 semi-simple $G$ is identified with the image of 
 the Cartan immersion $\iota: G/K  \hookrightarrow G$.
 For example, one obtains $\iota (\mathbb S^2) \subset \mathrm{SU}_2$ 
 and $\iota(\mathbb H^2) \subset \mathrm{SU}_{1, 1}$. 
 Thus as a consequence of Theorem \ref{thm:Alambda} 
 we can establish the loop 
 group formalism as for CMC surfaces in $\mathbb E^3$,
 spacelike CMC surfaces in Minkowski space and 
 in many other cases, \cite{Dorfmeister}. 
\end{Remark}

\subsection{Loop groups decompositions}
 In the rest of this subsection we will follow the procedure presented in \cite{DPW}, 
 even though our symmetric space $G = G \times G/ \Delta$ 
 is not necessarily compact and does, in general,  not even 
 carry any bi-invariant metric. But it has a ``nice'' bi-invariant 
 connection, and this will allow us to produce all 
 $\neutral$-harmonic maps into $G$ by the loop group method.

 Let $G$ be a Lie group which admits a faithful finite dimensional 
 representation and $G^\C$ its simply-connected linear complexification.
 Let $\LGC$ be the (connected)  loop group of  $G^{\mathbb{C}}$:
\be
 \LGC=\{\gamma:\mathbb{S}^1\to G^{\mathbb C}\}.
\ee
 We equip $\LGC$ with a weighted Wiener topology \cite{BD1} such that $\LGC$ is a Banach Lie group and all the subgroups occurring in this paper will be Banach Lie subgroups.
 Let $\mathcal{D}$ be the unit disk in the complex plane and 
 $\overline{\C}$ the extended complex plane.
 We now introduce the following subgroups of $\Lambda G^{\mathbb C} :$
\begin{eqnarray*}
 \LG &=&\{\gamma \in \Lambda G^{\mathbb C}
 \vert \ \gamma(\lambda)\in G\},\\
 \Lambda^{+}G^{\mathbb C} &=&\{\gamma \in \Lambda G^{\mathbb{C}}
 \ \vert \ \gamma \ \textrm{and}\ \gamma^{-1} 
 \textrm{extend holomorphically to}\ \mathcal{D}\},\\
\Lambda^{-}G^{\mathbb C} &=&\{\gamma \in \Lambda G^{\mathbb{C}}
 \ \vert \ \gamma \ \textrm{and}\ \gamma^{-1}
 \textrm{extend holomorphically to}\ \overline{\mathbb{C}}\setminus \mathcal{D}\},
 \\
  \Lambda^{-}_{*}G^{\mathbb C} &=&\{\gamma \in \Lambda^{-} G^{\mathbb{C}}
 \ \vert \ \gamma(\lambda=\infty)=\id \}.
\end{eqnarray*}
 Below we will recall two important decomposition 
 theorems obtained in \cite{BD1, Kellersch, {PreS:LoopGroup}}. 
\begin{Remark} In our applications we consider frames with values in 
 some real Lie group $G$ and choose a convenient complexification 
 (which may not necessarily be simply-connected). These frames 
 are continuous images of connected surfaces $M$ and attain the value $\id$ 
 at some base point $z_0 \in M$. Hence all these frames take values in 
 the connected component of  $\id \in \Lambda G^{\mathbb{C}}$. 
 Thus we are only interested in the decompositions of the identity component 
 of our loop group.
 If $\tilde{G}^\C$ denotes the simply-connected cover of $G^\C$, then the 
 loop group $\Lambda \tilde{G}^\C$ is connected and the canonical projection,
 induced from $\tilde{\pi} :\tilde{G}^\C \rightarrow G^\C$,  
 has as image the connected component of $\Lambda G^{\mathbb{C}}$. 
 Therefore, below we will write down the decomposition theorems for 
 $\Lambda \tilde{G}^\C$, but will apply them later to the projection  
 onto the connected component $(\Lambda G^\C)^o$. 
 It is not difficult to verify that the double cosets are parametrized by 
 the same set of representatives  and that  a double coset in $(\Lambda G^\C)^o$ 
 is open if and only if the corresponding double coset in $\Lambda G^{\mathbb{C}}$ 
 is open. It needs to be pointed out that the three references given above all use 
 simply-connected $G^\C$, even though in \cite{BD1} this was missed to state.
\end{Remark}
 By the Levi theorem 
 \cite[Theorem 18.4.3]{Hochschild}, 
 there exist a real analytic reductive subgroup $H$ of $G$, and 
 a simply-connected real analytic solvable subgroup 
 $B$ of $G$ such that   
 \begin{equation}\label{eq:Levi-decomposition}
 G \cong H \ltimes B.
 \end{equation}
 Note that $B$ is a normal subgroup of $G$, and 
 $B$ can be represented in the form  
 $B=A_1\ltimes ( A_2 \ltimes \cdots \ltimes N)$ 
 with simply-connected $1$-dimensional abelian 
 Lie groups $A_j$ and simply-connected  unipotent Lie group $N$. 
 Since $ G = H \cdot B  \cong  H \ltimes B$, the complexified  
 simply-connected group $G^{\mathbb C}$ satisfies 
 $G^{\mathbb C} =H^{\mathbb C} \cdot B^{\mathbb C}  \cong H^{\mathbb C} 
 \ltimes  B^{\mathbb C}$, we also have
\be
 \LGC \cong \LHC \cdot \LBC,\ \
  \Lambda^{-}G^{\mathbb C} \cong \Lambda^{-}H^{\mathbb C} \cdot 
 \Lambda^{-} B^{\mathbb C}, \ \
 \Lambda^{+}G^{\mathbb C} \cong \Lambda^{+}H^{\mathbb C} \cdot 
 \Lambda^{+} B^{\mathbb C}.
\ee
 Below we will state two decomposition theorems. More precisely, we will describe
 (to some extent) the (double) cosets of the action of  the product of some natural 
 subgroups of the loop group  $\Lambda G^{\mathbb C}$. 
 Actually, there are several natural choices. Since in this paper we are mainly interested in the open cosets, 
 these possibly different choices have little importance to us.

 Let $\Lambda^{d}H^{\mathbb C}$ denote the set of representatives in 
 the unique Birkhoff decomposition 
 of \cite[Corollary 5(c) ]{Kac-Pet} or \cite[Theorem 4.4]{BD1}
 of a simply-connected semi-simple Lie group
$H^\C$:
\begin{equation}
 \Lambda H^\C = \bigcup_{s\in \Lambda^{d}H^{\mathbb C}} 
(\Lambda^{-}_{\sharp} H^{\mathbb C})_s^{-} \cdot s \cdot \mathit C \Lambda^{+}_\sharp H^{\mathbb C},
\end{equation}
 where $\mathit C$ is  a Cartan subgroup of $G^\C$, 
 $\Lambda^{\epsilon}_\sharp H^{\mathbb C}$ denotes the 
 subgroup of $\Lambda^{\epsilon }H^{\mathbb C}$ defined such that 
 the  $\lambda$-independent term is in the maximal nilpotent subgroup of 
 the (relative to $\mathit C$ and a choice positive roots) 
 opposite Borel subgroup of $H^\C$, if $\epsilon = -$ 
 and in the maximal nilpotent subgroup of the Borel subgroup 
 (relative to $\mathit C$ and the same choice of positive roots)
 if $\epsilon = +$.
 Finally, $(\Lambda^{-}_{\sharp} H^{\mathbb C})_s^{-}$ is defined by 
\be
 (\Lambda^{-}_{\sharp} H^{\mathbb C})_s^{-} =
 \left\{
 h \in \Lambda^{-}_\sharp H^{\mathbb C}\;|\; 
 s^{-1} h s \in \Lambda^{-}_\sharp H^{\mathbb C}\right\}.
\ee
 Then note that 
 $\Lambda^{d}H^{\mathbb C}$ can be identified with the Weyl group of $\Lambda G^\C$.

 Further  define the subgroups
 \begin{equation*}
 \left(
 \Lambda^{-}_*B^{\mathbb C}
 \right)_{s}^{\pm}=\{b\in \Lambda^{-}_*B^{\mathbb C} \ \vert \ 
 sbs^{-1}\in \Lambda^{\pm}_{*}B^{\mathbb C}\},
\end{equation*}
 and set $\Lambda ^{\pm}_\sharp G^\C =  \Lambda^{\pm }_\sharp H^{\mathbb C} \cdot 
 \Lambda^{\pm} B^{\mathbb C}$ and 
 $(\Lambda ^{-}_{\sharp} G^\C)_s^- =  
 (\Lambda^{- }_{\sharp}H^{\mathbb C})_s^-  \cdot 
 (\Lambda^{-} B^{\mathbb C})_s^{-}$.
 Then from\cite[Theorem 4.5]{BD1}, \cite{Kellersch}, 
 we have:
 \begin{Theorem}[Birkhoff decomposition]
 \label{thm:Birkhoff}
 Assume $G^\C$ to be simply-connected, then the loop group $\LGC$ can be decomposed into  the disjoint union of double cosets$:$
\be
 \LGC = \bigcup_{s\in \Lambda^{d}H^{\mathbb C}} 
 (\Lambda^{-}_{\sharp}G^{\mathbb C})_s^- \cdot s\left(
 \Lambda^{-}_*B^{\mathbb C}
 \right)_{s}^{+} \cdot  \mathit C \Lambda^{+}_\sharp G^{\mathbb C}.
\ee
 Moreover, the subset 
 $\mathit {Br}_{G^\mathbb{C}}:=\Lambda^{-}_{*}G^{\mathbb C} \cdot 
 \Lambda^{+} G^{\mathbb C}$
 is called the {\rm (left) Birkhoff big cell}  of $\Lambda G^{\mathbb {C}}$ and it is 
 an open and  dense subset of $\Lambda G^{\mathbb{C}}$. The multiplication map 
\be
\Lambda^{-}_{*}G^{\mathbb C} \times 
 \Lambda^{+} G^{\mathbb C}
 \to 
\Lambda^{-}_{*}G^{\mathbb C} \cdot 
 \Lambda^{+} G^{\mathbb C}
\subset \LGC
\ee
 provides a complex analytic diffeomorphism onto the Birkhoff big cell 
 $\mathit {Br}_{G^\mathbb{C}}.$
 \end{Theorem}
\begin{Remark}
 As in \cite{BD1}  one can easily show that each element $g \in \Lambda G^\C$ 
 can be represented uniquely in the form
\begin{equation}
 g = (h_- s b_-^- s^{-1}) (sb_-^+) (b_+ h_+),
\end{equation}
where $s \in  \Lambda^{d}H^{\mathbb C}$, 
 $b_-^+ \in  \left( \Lambda^{-}_*B^{\mathbb C} \right)_{s}^{+}$,
 $b_-^- \in \left( \Lambda^{-}_*B^{\mathbb C} \right)_{s}^{-}$, 
 $b_{+} \in  \Lambda^{+} B^{\mathbb C}$, $h_{+} \in  \Lambda^{+}_{\sharp} H^{\mathbb C}$
 and $h_- \in (\Lambda^-_{\sharp} H ^\C)_s^{-}$.
 From this it is not difficult to show that exactly 
 one double coset is open, namely the one with $s = \id$.
\end{Remark}
 In a similar fashion we obtain an Iwasawa decomposition of $\Lambda G^\C$.
 First we consider an Iwasawa decomposition (with disjoint cosets) of 
 $H^\C$ as derived in \cite[Chap. 4]{Kellersch}, \cite[Theorem 6.1]{BD1}:
\begin{equation}
\Lambda H^\C = \bigcup_{s\in \Lambda^{m}H^{\mathbb C}}\Lambda H \cdot 
s \cdot \Lambda^+ H^\C.
\end{equation}
 Here $\Lambda^{m}H^{\mathbb C}$ is a specific set of representatives
 for the double coset given in \cite{OV}.
 Then from \cite[Theorem 6.5]{BD1}, we have:
 \begin{Theorem}[Iwasawa decomposition]
 \label{thm:Iwasawa}
 Let $G^\C$ be simply-connected, then the loop group  
 $\LGC$ can be decomposed into a disjoint union 
 of double cosets$:$
\be
 \LGC=\bigcup_{s\in \Lambda^{m}H^{\mathbb C}}
 \LG \cdot
 s \Lambda B
 \cdot
 \Lambda^{+}G^{\mathbb C}.
\ee
 Moreover,  the subset
 $\mathit {Iw}_{G}^{\id}:=\Lambda{G}\cdot \Lambda^{+}G^{\mathbb C}$ 
 is called the 
 {\rm Iwasawa big cell} $($containing the identity element $\id$$)$.
 It is an open set in $\Lambda{G}^{\mathbb C}$. 
\end{Theorem}
\begin{Remark}
\mbox{}
\begin{enumerate}
\item One can describe a unique decomposition of each loop \cite{BD1}, 
 but we will not need these details for this paper.

\item In general, there  are several open Iwasawa cells $\mathit {Iw}^{\omega}$. 
 In this paper we will mostly use  \textit{the Iwasawa big cell}.

\item Of course, one is interested to know when there is only one open Iwasawa cell. This happens if and only if 
     $\mathit {Iw}^{\id}$ is open and dense in $\LGC$.
\end{enumerate}
\end{Remark}
 About the denseness of $\mathit {Iw}^{\id}$ the following result is 
 fundamental.
 \begin{Theorem}[Theorem 7.2 and 7.3 in \cite{BD1}] Let $G^\C$ be simply-connected. Then,
 if the semi-simple part of a maximal 
 compact subgroup of $H^{\mathbb C}$ is simply-connected, then 
 the Iwasawa big cell $\mathit {Iw}^{\id}_H$ is dense in 
 $\Lambda{H}^\mathbb{C}$ and we have 
\be
 \mathit {Iw}^{\id}_G=\mathit {Iw}^{\id}_{H}\cdot \Lambda{B}^{\mathbb C}.
\ee
 In particular, $\mathit {Iw}^{\id}_G=\Lambda{G}^{\mathbb C}$ if and only if 
 the semi-simple part $S$ of $H$ is compact.
 \end{Theorem}

\subsection{Generalized Weierstrass type representation} 
 We denote by $\G$ the direct product $G \times G $ 
 and $(G \times G)^{\C} =G^{\C} \times G^{\C}$ by $\G^{\C}$.
 Consider the double loop group
\be
 \LGCD := (\Lambda G \times \Lambda G)^{\C} = \LGC \times \LGC.
\ee
 Then the \textit{twisted loop group} $\LGCD_{\sigma}$, twisted by $\sigma$
\eqref{sigma}, and its real form $\LGD_{\sigma}$ are
 defined by
\be
 \LGCD_{\sigma} := \{ (g(\l), g(-\l))\;|\; g \in \LGC\}, \;\;
 \LGD_{\sigma} := \{ (g(\l), g(-\l))\;|\; g \in \LG \}.
\ee
 According to the Levi decomposition $G \cong H \ltimes B$, we consider the twisted 
 double loop groups of $B^{\mathbb C}$ and $H^{\mathbb C}$, and its real forms:
 \begin{align*}
  \Lambda \B^{\mathbb C}_{\sigma}  &= \left\{ (b(\l), b(-\l))\; |\; b \in \Lambda B^{\mathbb C}\right\}, \;\;\;
 \Lambda \mathcal H^{\mathbb C}_{\sigma} = \{
 (h(\l), h(-\l)) \;|\; h \in \Lambda H^{\mathbb C}\}, \\
  \Lambda \B_{\sigma}  &= \left\{ (b(\l), b(-\l))\; |\; b \in \Lambda B\right\}, \;\;\;
 \Lambda \mathcal H_{\sigma} = \{
 (h(\l), h(-\l)) \;|\; h \in \Lambda H\}.
 \end{align*}
 For the ``positive loop subgroup'' and ``negative loop subgroup'' 
 of $\LGCD_{\sigma}$ we set:
\be
 \PNLGCD_{\sigma} = \{(g(\l), g(-\l))\;|\;g \in \PNLGC \}, \;\;\;
\ee
 and we will need to define the following subgroup: 
\begin{align*}
\NLGCDN_{\sigma}&=\{(g(\l), g(-\l)) \;|\; \mbox{$g \in \NLGC$ and
 $g(\l = \infty) =e$}\}.
\end{align*}
 According to the Birkhoff decomposition of $\LGC$ in Theorem \ref{thm:Birkhoff}, 
 we introduce the following subgroups of $\LGCD_{\sigma}$:
\begin{align*}
\Lambda_{\sharp}^{\pm} \G^{\mathbb C}_{\sigma}
& =\left\{ (g(\l), g(-\l)) \;|\; g \in 
 \Lambda_{\sharp}^{\pm} G^{\mathbb C} \right\},\\
(\Lambda_{\sharp}^{-} \G^{\mathbb C}_{\sigma})_s^{-}
& =\left\{ (g(\l), g(-\l)) \;|\; g \in 
 (\Lambda_{\sharp}^{-} G^{\mathbb C})_s^{-}\ \right\}, \\
(\Lambda_{*}^{-} \mathcal B^{\mathbb C}_{\sigma})_s^{+}
& =\left\{ (g(\l), g(-\l)) \;|\; g \in 
 (\Lambda_{*}^{-} B^{\mathbb C}_{\sigma})_s^{+} \right\}.
\end{align*}
 Now we consider the Birkhoff decomposition as well as the
 Iwasawa decomposition for the loop group
 $\LGCD_{\sigma} (\cong \LHCD_{\sigma} \cdot \LBCD_{\sigma})$.
 Then the Birkhoff and Iwasawa decomposition Theorems \ref{thm:Birkhoff} 
 and \ref{thm:Iwasawa} for $\Lambda{G}^\mathbb{C}$ induce the following 
 decomposition theorems for $\LGCD_{\sigma}$. 
 \begin{Theorem}[Birkhoff decomposition]
 \label{thm:Birkhoffdouble}
 The loop group $\LGCD_{\sigma}$ 
 can be decomposed into a disjoint union of double cosets:
\be
 \LGCD_{\sigma} = \bigcup_{s \in \Lambda^d \mathcal H^{\mathbb C}_{\sigma}}
 (\Lambda_{\sharp}^{-} \G^{\mathbb C}_{\sigma})_s^{-} \cdot  
 s(\Lambda_{*}^{-} \mathcal B^{\mathbb C}_{\sigma})_s^{+}
 \cdot \mathcal C \Lambda_{\sharp}^{+} \G^{\mathbb C}_{\sigma},
\ee
 where $\mathcal  C = (\mathit C, \mathit C)$ with a Cartan subgroup $\mathit C$ 
 and $\Lambda^d \mathcal H^{\mathbb C}_{\sigma} = \{(h(\l), h(-\l))\;|\;
 h \in \Lambda^d H^{\mathbb C}\}$.
 Moreover, the subset $\mathcal Br_{\mathcal G^{\mathbb C}} = 
  \Lambda_{*}^{-} \G^{\mathbb C}_{\sigma} 
 \cdot \Lambda^{+} \G^{\mathbb C}_{\sigma}$ 
 is called the {\rm (left) Birkhoff big cell} and it is an open and dense subset of 
 $\LGCD_{\sigma}$. The multiplication map 
\be
\Lambda_{*}^{-} \G^{\mathbb C}_{\sigma} 
 \times \Lambda^{+} \G^{\mathbb C}_{\sigma}
 \to    
\Lambda_{*}^{-} \G^{\mathbb C}_{\sigma} 
 \cdot \Lambda^{+} \G^{\mathbb C}_{\sigma}
 \subset \LGCD_{\sigma}
\ee
 provides an analytic diffeomorphism 
 onto  the Birkhoff big cell  $\mathcal B_{\mathcal G^{\mathbb C}}$.
\end{Theorem}
\begin{Remark}
As always, there is only one open (and dense) Birkhoff big cell.
\end{Remark}
 Then we have the following Iwasawa decomposition theorem.
\begin{Theorem}[Iwasawa decomposition]\label{thm:Iwasawadouble}
 The loop group $\LGCD_{\sigma}$ 
 can be decomposed into a disjoint union 
 of double cosets:
\be
 \LGCD_{\sigma}=\bigcup_{s\in \Lambda^{m}\H^{\mathbb C}}
 \Lambda \G_{\sigma}  \cdot
 s \Lambda \B
 \cdot
 \Lambda^{+}\G^{\mathbb C}_{\sigma}.
\ee
 Here $ \Lambda^{m}\H^{\mathbb C} 
 = \{(h(\l), h(-\l))\;|\; h \in \Lambda^m H^{\mathbb C}\}$.
 Moreover,  the subset
 $\mathcal Iw_{\G}^{\id}:= \Lambda{\G}_{\sigma}\cdot \Lambda^{+}\G^{\mathbb C}_{\sigma}$ 
 is called the 
{\rm Iwasawa big cell} and it
 is an open set in $\Lambda{\G}^{\mathbb C}_{\sigma}$. 
\end{Theorem}
\subsection{}
 From now on, we consider only the connected component of 
 a twisted loop group and denote it by the same symbol.
 Now we can apply the usual loop group scheme. 
 Starting from a $\neutral-$harmonic map $\varphi$, we consider the corresponding 
 extended frame $\F_\lambda$.
 Where possible we perform a Birkhoff decomposition of the extended frame $\F_\lambda$ 
 as described in Theorem \ref{thm:Birkhoffdouble}:
\begin{equation}\label{Birkhoff-split}
 \F_\lambda = \F_{-} \V_{+}, 
\end{equation}
 where $\F_{-} \in \NLGCDN_{\sigma}$ and $\V_{+} \in \PLGCD_{\sigma}$.
 Then the usual argument shows the following, see \cite{DPW}.
\begin{Theorem}\label{thm:Normalizedpot}
 Let $\varphi: \D \rightarrow G$ be a $\neutral$-harmonic map and 
 $\F_\lambda$ its extended frame. 
 Moreover let $\F_\lambda = \F_{-} \V_{+}$
 be the Birkhoff decomposition given in \eqref{Birkhoff-split}.
 Then the following statements hold{\rm :}
\begin{enumerate}
\item There exists a discrete subset $S \subset \D$ such that 
 \eqref{Birkhoff-split} is defined for all $z \in \D \setminus S.$

\item The map $\F_{-}$ only depends on $z$. It is a meromorphic 
$\NLGCDN_{\sigma}$-valued matrix function.

\item The Maurer-Cartan form $\N = \F_{-}^{-1} d \F_{-}$ of $\F_-$
 has the form
 \begin{equation}\label{eq:normalizedpot}
 \N(z, \l) = \F_{-}(z, \l)^{-1} d \F_{-}(z, \l)= 
 \left(\l^{-1}\xi(z),  -\l^{-1}\xi(z)\right),
 \end{equation}
 where $\xi$ is a meromorphic $1$-form on $\D$ with values in 
 $\mathfrak g^{\C}$.
\end{enumerate}
\end{Theorem}
\subsection{}
 The converse procedure is as follows:
 Consider a meromorphic $1$-form $\N$ on $\D$ of the form stated  
 in \eqref{eq:normalizedpot}.

 {\bfseries Step 1.} Solve the pair of ordinary differential equations:
\be
 d \R_{-} = \R_{-}\> \N
\ee
 with any initial condition $\R_{-}(z_0) \in \G^{\C} (= G^{\C} \times G^{\C})$
 at some base point $z_0 \in \D$  and assume that the solution 
 $\R_{-}$  is  meromorphic and takes values in $\LGCD$. 

 {\bfseries Step 2.} From the Iwasawa decomposition in Theorem 
 \ref{thm:Iwasawadouble} we obtain (for all $z \in \D$ for which $\R_{-} \in \tilde{\P}_{\G}$):
\be
 \R_{-} = \F \>\W_{+}, \;\;(\F \in \LGD_{\sigma}, \; \W_+ \in \PLGCD_{\sigma}),
\ee
 where $\F$ and $\W_+$ have the form
\be
 \F(z, \bar z,  \l)= \left(F(z, \bar z, \l), F(z, \bar z, -\l)\right), \;\;
 \W_+(z, \bar z,  \l)= \left(W_+(z, \bar z, \l), W_+(z, \bar z, -\l)
 \right).
\ee 
 There is freedom in this decomposition. 
 Setting 
\be
 \hat \F(z, \bar z, \l) = (F(z, \bar z, \l) F(z, \bar z, 1)^{-1}, \;
 F(z, \bar z, - \l) F(z, \bar z, 1)^{-1}),
\ee
 we obtain the following Iwasawa decomposition, where now 
\be
 \R_- = \hat \F \>\hat \W_+ \;\;\mbox{and} \;\;
 \hat \F(z, \bar z, 1)= (e, \varphi).
\ee
 In this case, $\hat \A_{\l} = \hat \F^{-1} d \hat \F$ has 
 the form \eqref{eq:Alambda} and  for $\l =1$ 
 we obtain $\hat \A_{\l =1} 
 = (0, \alpha)$. Hence $\hat \A_{\l}$ is of the form \eqref{eq:Alambda2}.
 Setting 
 $\hat \F(z, \bar z, \l) = (\hat F(z, \bar z, \l),  \hat F(z, \bar z,- \l))$ 
 we know $\hat F(z, \bar z, \l= 1) = \mathrm{id}$
 and $\hat F(z, \bar z, \l =-1) = \varphi(z, \bar z)$. 
\begin{Theorem}
 The map $\varphi (z, \bar z): = \hat F(z, \bar z, -1)$ 
 is $\neutral$-harmonic. Moreover, 
\be
 \hat \varphi(z, \bar z, \l) = \hat F(z, \bar z, -\l) 
 \hat F(z, \bar z, \l )^{-1}
\ee
 is a harmonic map into $(G, \neutral)$ for all 
 $\lambda \in \mathbb S^1$. 
 The map $\hat{F}(z,\bar{z},\l)$ is an extended solution of $\varphi$.
\end{Theorem}
 For many purposes it is very convenient to start the construction scheme with 
 a holomorphic potential (as opposed to a meromorphic potential).
 In this case one will need to admit more (usually infinitely many) powers of 
 $\lambda$ in a Fourier expansion of the potential. The construction scheme just outlined 
 above works verbatim in the same way
 and produces a $\neutral$-harmonic map into $G$.
 To make sure that starting from some holomorphic potential we do not miss 
 any harmonic maps we state that the proof given in the appendix of \cite{DPW}
 can be carried out mutatis mutandis and we obtain:
\begin{Theorem}
 Let $\mathbb{D}$ be a non-compact simply-connected Riemann surface $\D$, 
 then for every harmonic map $\varphi$
 from $\mathbb{D}$ to a Lie group $G,$
 equipped with the neutral connection $\neutral$, 
 there exists a holomorphic potential defined on $\mathbb{D}$
 which generates this harmonic map $($actually an 
 $\mathbb S^1$-family of $\neutral$-harmonic 
 maps $\varphi_{\lambda}$ with  $\varphi = \varphi_{\lambda =1}$$)$
 by the construction scheme outlined above.
\end{Theorem}
\section{Examples}\label{sc:Examples}
 In this section, as an application of Section \ref{sc:DPW}, 
 we give the generalized Weierstrass type representation for 
 $\neutral$-harmonic maps into the $3$-dimensional solvable Lie groups. 
 We first remark that 
 the case of the $3$-dimensional Heisenberg group $\mathrm{Nil}_3$ 
 was discussed in \cite{BD} in detail. 
 Thus we omit this case.
\subsection{$3$-dimensional solvable Lie groups}
 Let us consider the $2$-parameter family of $3$-dimensional solvable Lie groups 
 $G(\mu_1, \mu_2)$ given in \eqref{eq:solvable}. 
 For the loop group $\Lambda G(\mu_1, \mu_2)^{\C}$, 
 the respective Birkhoff and Iwasawa decompositions 
 in Theorem \ref{thm:Birkhoff} and \ref{thm:Iwasawa} 
 are given explicitly  as follows.
\begin{Theorem}\label{thm:BirkhoffGmu}
 The Birkhoff decomposition of $\Lambda G(\mu_1,\mu_2)^{\mathbb{C}}$ is 
 given by 
\be
\Lambda G(\mu_1,\mu_2)^{\mathbb{C}}
=
\Lambda^{-}_{*} G(\mu_1,\mu_2)^{\mathbb{C}}\cdot
\Lambda^{+} G(\mu_1,\mu_2)^{\mathbb{C}}.
\ee
 Every element 
\be
 g(\l) = 
 \begin{pmatrix}
 e^{\mu_1 x^3(\l)} & 0 & x^1(\l)  \\
 0 & e^{\mu_2 x^3(\l)}  & x^2(\l)  \\
 0 & 0 & 1
 \end{pmatrix}
\ee
 of $\Lambda G(\mu_1,\mu_2)^{\mathbb{C}}$
 is globally decomposed as
\be
 g(\lambda)= g_{-} (\l) g_{+} (\l),
\ee
where 
\be
\;\;
g_{\pm}(\l) =
\begin{pmatrix}
 e^{\mu_1 x_{\pm}^3 (\l)} & 0 & x_{\pm}^1 (\l) \\
 0 & e^{\mu_2 x_{\pm}^3 (\l)} & x^2_{\pm} (\l) \\
 0 & 0 & 1 
\end{pmatrix}.
\ee
 Here using 
 the expansion $x^3(\l) = \sum_{j=-\infty}^{\infty} x_j^3 \l^j$,
 the functions $x^3_{\pm}(\l)$ are given by
\be
 x^{3}_{+}(\l) =\sum_{j\geq 0}x^{3}_{j}\l^j,
 \quad x^{3}_{-}(\l)=\sum_{j< 0}x^{3}_{j}\l^{j},
\ee
 and  using the expansion
\be
 \exp\left(- \mu_k x^3_{-}(\l) \right) x^k(\l)
 = \sum_{j=-\infty}^{\infty} \hat x_j^k \l^j\;\;(k =1, 2),
\ee
 $x^k_{\pm}(\l), \;(k=1, 2)$ are given by
\be
\left\{
\begin{array}{l}
 x^{k}_{+}(\l) =\sum_{j\geq 0}\hat x^{k}_{j}\l^{j},\;\;(k=1, 2), \\[0.1cm]
 x^{k}_{-}(\l)=(\sum_{j< 0}\hat x^{k}_{j}\l^{j})\exp(\mu_k x^3_{-}(\l)), \;\;(k=1, 2).
\end{array}
\right.
\ee
\end{Theorem}
\begin{Theorem}\label{thm:Iwasawa-sol}
 The Iwasawa decomposition of $\Lambda G(\mu_1,\mu_2)^{\mathbb{C}}$ is 
 given by 
\be
\Lambda G(\mu_1,\mu_2)^{\mathbb{C}}
=
\Lambda G(\mu_1,\mu_2)\cdot
\Lambda^{+} G(\mu_1,\mu_2)^{\mathbb{C}}.
\ee
 Every element
\be
 g(\l) = 
 \begin{pmatrix}
 e^{\mu_1 x^3(\l)} & 0 & x^1(\l)  \\
 0 & e^{\mu_2 x^3(\l)}  & x^2(\l)  \\
 0 & 0 & 1
 \end{pmatrix}
\ee
 of $\Lambda G(\mu_1,\mu_2)^{\mathbb{C}}$ is globally decomposed as
\be
 g(\l)= \tilde g(\l)g_{+}(\l),
\ee
 where 
\be
 \tilde g(\l) =
 \begin{pmatrix}
 e^{\mu_1 \tilde x^3(\l)} & 0 & \tilde x^1(\l)  \\
 0 & e^{\mu_2 \tilde x^3(\l)}  & \tilde x^2(\l)  \\
 0 & 0 & 1
 \end{pmatrix}, \;\;
 g_+ (\l) =
 \begin{pmatrix}
 e^{\mu_1 x_+^3(\l)} & 0 & x_+^1(\l)  \\
 0 & e^{\mu_2 x_+^3(\l)}  & x_+^2(\l)  \\
 0 & 0 & 1
 \end{pmatrix}.
\ee
 Here using the expansion $x^3(\l) = \sum_{j=-\infty}^{\infty} x_j^3 \l^j$, 
 the functions $\tilde x^3(\l)$ and $x^3_+(\l)$ are given by,
\be
x^{3}_{+}(\l)
=\sum_{j\geq 0}x^{3}_j\l^j
-\sum_{j<0}\overline{x^{3}_j}\>\l^{-j}, \;\;\;
 \tilde x^{3}(\l)=\sum_{j<0} \left(x^3_j \l^j +\overline{x^3_j} \l^{-j}\right),
\ee
 and using the expansions of 
\be
 \exp\left(- \mu_k \tilde x_3(\l)\right) x^k(\l) 
 =  \sum_{j=-\infty}^{\infty} \hat x_j^k \l^j,
 \;\;\;(k =1, 2), 
\ee
 $\tilde x^k(\l)$ and $x^k_+(\l),\;(k =1, 2)$ are given by
\be
 x^{k}_{+}(\l)=\sum_{j\geq  0} \hat x^k_j\l^j-\sum_{j<0}\overline{\hat x^k_j}\l^{-j},
 \;\;\;
 \tilde x^{k}(\l)=\sum_{j< 0} \left(\hat x^k_j \l^j+\overline{\hat x^k_j}\l^{-j}\right)
 e^{\mu_k \tilde x^3(\l)}, 
 \;\;\;(k =1, 2).
\ee
\end{Theorem} 
\begin{Remark}
 It is shown in \cite[Lemma 4.3 and Lemma 6.4]{BD1},
 the Birkhoff and Iwasawa decompositions for the loop group $\Lambda B^{\C}$ 
 of a simply-connected solvable Lie group $B^{\C}$ is {\it global}, 
 that is, the Birkhoff and the Iwasawa decompositions are given respectively as follows:
\begin{align*}
 \Lambda B^{\C} = \Lambda^{-}_* B^{\C} \cdot \Lambda^+ B^{\C},\;\;\;\;
 \Lambda B^{\C} = \Lambda B \cdot \Lambda^+ B^{\C}.
\end{align*}
\end{Remark}
 Let $\varphi$ be a $\neutral$-harmonic map parametrized as 
\be
 \varphi (z, \bar z) = 
 \begin{pmatrix}
  e^{\mu_1 \varphi^3(z, \bar z)} & 0 &\varphi^1(z, \bar z) \\
  0 & e^{\mu_2 \varphi^3(z, \bar z)} &\varphi^2(z, \bar z) \\
  0 & 0 & 1
 \end{pmatrix} : \mathbb D \to G(\mu_1, \mu_2) \subset \mathrm{GL}_3 \mathbb R.
\ee
 Here we assume that $\mathbb D$ is a simply-connected domain in $\mathbb C$ containing
 $0$ and $z \in \mathbb D$ is a conformal coordinate.
 Then it is easy to see that  
 \begin{equation}\label{eq:MCvarphi}
 \varphi^{-1} d \varphi = \alpha = \alpha^{\prime} + \alpha^{\prime \prime}
 = \begin{pmatrix}
  \mu_1 \varphi^3_z & 0 & e^{-\mu_1 \varphi^3} \varphi^1_z \\
  0 & \mu_2 \varphi^3_z & e^{-\mu_2 \varphi^3} \varphi^2_z \\
  0 & 0 & 0 
 \end{pmatrix} \> dz 
+ 
 \begin{pmatrix}
  \mu_1 \varphi^3_{\bar z} & 0 & e^{-\mu_1 \varphi^3} \varphi^1_{\bar z} \\
  0 & \mu_2 \varphi^3_{\bar z} & e^{-\mu_2 \varphi^3} \varphi^2_{\bar z} \\
  0 & 0 & 0 
 \end{pmatrix} \> d{\bar z}.
\end{equation}
 From Corollary \ref{coro:affharmonic}, $\varphi(z, \bar z)$ is $\neutral$-harmonic if and only if $\bar \partial \alpha^{\prime} - \partial \alpha^{\prime \prime}=0$ which is 
 equivalent to 
\be
 2 \varphi^j_{z \bar z} - \mu_j (\varphi^j_z \varphi^{3}_{\bar z} 
 + \varphi^j_{\bar z} \varphi^3_z) =0 \;\;(j = 1, 2),\;\;
 \varphi^3_{z \bar z} =0.
\ee
 We now set $\A = (0,\alpha)$ and $\A_{\K} = \frac{1}{2}(\alpha, \alpha)$, which takes 
 values in
 the fixed point set of the derivative of the involution $\sigma (a, b) = (b, a)$ for $
 (a, b) \in G(\mu_1, \mu_2) \times G(\mu_1, \mu_2)$, and the complement $\A_{\P} = 
 \frac{1}{2}(-\alpha, \alpha)$, that is, $\A = \A_{\K} + \A_{\P}$. Moreover 
 we decompose $\A_{\P}$ into its $(1,0)$ and $(0,1)$-parts as 
 $\A_{\P} = \A_{\P}^{\prime} + \A_{\P}^{\prime \prime}  = \frac{1}{2}
  ( -\alpha^{\prime}, \alpha^{\prime}) 
 + \frac{1}{2}  ( -\alpha^{\prime \prime}, \alpha^{\prime \prime})$
 and define $\A_{\l}$ as 
\be
 \A_{\l} = \l^{-1} \A_{\P}^{\prime} + \A_{\K} + \l \A_{\P}^{\prime \prime}.
\ee
 Since $\varphi$ is $\neutral$-harmonic, by Theorem \ref{thm:Alambda}, there exists a 
 $\F_{\l}$, which is a solution to the equation $\F_{\l}^{-1} d \F_{\l} = \A_{\l}$.  
 From now on, for convenience, we choose the base point to be $(0, 0)$  and there we assume 
 \begin{equation}\label{eq:initialzbarz}
 \F_{\l}|_{(z, \bar z) =(0,0)} = (\id,\; \id).
 \end{equation}
 Then, decomposing $\F_{\l}$ by the Birkhoff decomposition Theorem \ref{thm:BirkhoffGmu}, 
 we obtain  
 \begin{equation}\label{eq:BirkofF}
 \F_{\l} = \F_{-} \V_{+}.
 \end{equation}
 From Theorem \ref{thm:Normalizedpot}, we know that $\F_{-}$ depends only on $z$ and moreover 
 it is meromorphic on $\D$. 
 Then a direct computation shows the following theorem.
\begin{Theorem}[The normalized potentials]
 Let $\F_{-}$ be the loop defined in \eqref{eq:BirkofF}.
 Then $\F_{-}$ is meromorphic on $\D$ and the pair of normalized potentials 
\be
 \N_{-}(z, \l) = \F_{-}(z, \l)^{-1} d \F_{-}(z, \l) 
 = \left(\l^{-1}\xi(z), \;\;-\l^{-1}\xi(z)\right)
\ee
 is determined by
\begin{equation}\label{eq:xi}
 \xi(z) =  
 \begin{pmatrix}
  \mu_1 \xi^3(z) & 0 & \xi^1(z) \\
  0 & \mu_2 \xi^3(z) & \xi^2(z) \\
  0 & 0 & 0 \\
 \end{pmatrix} dz,
\end{equation}
 where $\xi^1, \xi^2$ and $\xi^3$ are meromorphic functions on $\D$ given
\begin{align*}
\xi^{1}(z) &=  -\frac{1}{2}e^{-\frac{1}{2} \mu_1 \varphi^3 (z, 0)} 
 \varphi^1_z(z, 0) +\frac{\mu_1 \varphi^3_z(z,0)}{8 \pi i}
 \left.\int_{\mathbb D} \frac{ e^{- \frac{1}{2} \mu_1 \varphi^3(\xi, \bar \xi)}
 \partial_{\bar \xi} \varphi^1(\xi, \bar \xi)}{\xi-z} \> d\xi \wedge d \bar \xi
 \>\right|_{\bar z=0}, \\
\xi^{2}(z) &=  -\frac{1}{2}e^{-\frac{1}{2} \mu_2 \varphi^3 (z, 0)} \varphi^2_z(z, 0)
 +\frac{\mu_2 \varphi^3_z(z,0) }{8 \pi i}
 \left. \int_{\mathbb D} \frac{ e^{- \frac{1}{2} \mu_2 \varphi^3(\xi, \bar \xi)}\partial_{\bar \xi} \varphi^2(\xi, \bar \xi)}{\xi-z} \> d\xi \wedge d \bar \xi
 \>\right|_{\bar z=0}, \\
\xi^3(z) & =- \frac{1}{2} \varphi^3_{z}(z,0).
\end{align*}
\end{Theorem}
\begin{proof}
 Let $F$ be the extended solution given by $\F_{\l} = (F(\l), F(-\l))$ and 
 $V_+$ the plus element given by $\V_+ =(V_+(\l), V_+(-\l))$.
 Then the Maurer-Cartan form for $\F_-$ can be computed as 
\begin{align}
 \F_-^{-1} d \F_{-} &= \V_+ \F_{\l}^{-1} d \F_{\l} \V_+- d \V_+ \V_+^{-1} \nonumber \\
 & = \left(\ad (V_+(\l)) (\alpha_{\l}) - d V_+(\l) V_+(\l)^{-1}, \; s(\l)
      \right),\label{eq:righthand}
\end{align}
 where the second component $s(\l)$ is just the first component 
 with the replacement $\l$ to $-\l$. Thus we only consider the first component.
 Since the left hand side is a meromorphic $1$-form on $\D$, the same is true for the 
 right hand side. 
 Using $\alpha_{\l} = \frac{1}{2}(1- \l^{-1}) \alpha^{\prime} + 
 \frac{1}{2}(1- \l) \alpha^{\prime \prime}$, we can rephrase the first 
 component of the right hand side of \eqref{eq:righthand} as 
\begin{equation}\label{eq:alphagauge}
 -\frac{\l^{-1}}{2} \ad (V_+) \alpha^{\prime} + 
 \frac{1}{2} \ad (V_+) \alpha^{\prime} - \partial_{z}V_+ V_+^{-1} \>dz, 
 \;\;\;\mbox{(the $(1, 0)$-part)},
\end{equation}
 and 
\begin{equation}\label{eq:alphagauge2}
 - \frac{\l }{2}\ad (V_+) \alpha^{\prime \prime} + 
 \frac{1}{2} \ad (V_+) \alpha^{\prime \prime} - \partial_{\bar z}V_+ V_+^{-1} \>d\bar z, 
 \;\;\;\mbox{(the $(0, 1)$-part)}.
\end{equation}
 Since the normalized potential  $\F_{-} d \F_{-}$ is the $\l^{-1}$-term of 
 the $(1, 0)$-part, we can compute it by 
 using $V_+ = V_{+0} + V_{+1} \l + \cdots$;
\be
 \F_-^{-1} d \F_{-} =\left(-\frac{\l^{-1}}{2} \ad (V_{+0}) \alpha^{\prime}, \;\;
\frac{\l^{-1}}{2} \ad (V_{+0}) \alpha^{\prime}\right).
\ee
 Set $V_{+}$ as 
\be
 V_{+}(z, \bar z, \l) =
 \begin{pmatrix} 
  e^{\mu_1 v^3_+(z, \bar z, \l)} & 0 &v^1_+(z, \bar z, \l) \\
  0 & e^{\mu_2 v^3_+(z, \bar z, \l)} & v^2_+(z, \bar z, \l)\\
  0 & 0 &1
 \end{pmatrix}.
\ee
 Then using the matrix form of $\alpha^{\prime}$ given in \eqref{eq:MCvarphi}, 
 we can rephrase \eqref{eq:alphagauge} without 
 the term $-\frac{1}{2}\l^{-1}\ad (V_+) \alpha^{\prime}$  as
\be
\left\{
 \begin{array}{l}
 \mu_k(\frac{1}{2} \varphi^3_z - (v^3_+)_z) dz \;\;\mbox{for the $(k, k)$-entry $(k=1, 2)$}, \\[0.1cm]
 \left(\frac{1}{2}e^{\mu_k(v^3_+ -\varphi^3)} \varphi^k_z
 -\mu_k(\frac{1}{2} \varphi^3_z - (v^3_+)_z) v^k_+  -  (v^k_+)_z\right) dz
 \;\;\mbox{for the $(k, 3)$-entry $(k=1, 2)$}, \\
 \end{array}
\right.
\ee
 and zero for the other $(k, \ell)$-entries.
 Since $\F_-$ takes values in $\NLGCDN$, thus $\frac{1}{2} \varphi^3_z - (v^3_{+0})_z$
 needs to vanish,  where $v^3_+ = v^3_{+0} + v^3_{+1} \l +  v^3_{+2}\l^2 + \cdots$.
 Moreover since $\F_-$ is also a meromorphic function on $\D$, that is, 
 it is independent of 
 $\bar z$. Thus setting $\bar z =0$ on $\frac{1}{2} \varphi^3_z - (v^3_{+0})_z$, we have 
 $\frac{1}{2} \varphi^3_z(z, 0) - (v^3_{+0})_z(z, 0) =0$.
 Finally, using the initial condition $\varphi^{3}(0, 0) = v^3_{+0}(0, 0) =0$, we obtain 
\begin{equation}\label{eq:Cauchydata}
 v^3_{+0}(z, 0) = \frac{1}{2} \varphi^3(z, 0).
\end{equation}
 We now compute the $(0, 1)$-part of \eqref{eq:righthand}.
 Then a similar computation shows that 
 using the matrix form of $\alpha^{\prime \prime}$ given in \eqref{eq:MCvarphi}, 
 we can rephrase  \eqref{eq:alphagauge2} without 
 the term $-\frac{1}{2} \l\ad (V_+) \alpha^{\prime}$  as
\begin{equation}\label{eq:antialphamatrix}
\left\{
 \begin{array}{l}
 \mu_k(\frac{1}{2} \varphi^3_{\bar z} - (v^3_+)_{\bar z}) d\bar z 
  \;\;\mbox{for the $(k, k)$-entry $(k=1, 2)$}, \\[0.1cm]
 \left(\frac{1}{2}e^{\mu_k(v^3_+ -\varphi^3)} \varphi^k_{\bar z}
 -\mu_k(\frac{1}{2} \varphi^3_{\bar z} - (v^3_+)_{\bar z}) v^k_+  -  (v^k_+)_{\bar z}\right) d\bar z
 \;\;\mbox{for the $(k, 3)$-entry $(k=1, 2)$}, \\
 \end{array}
\right.
\end{equation}
 and zero for the other $(k, \ell)$-entries.
 The constant term of the $(1, 1)$-entry in \eqref{eq:antialphamatrix}
 needs to  vanish, that is,
 $(v_{+0}^3)_{\bar z}(z, \bar z) = \frac{1}{2} \varphi^3_{\bar z}(z, \bar z)$.
 Thus using \eqref{eq:Cauchydata}, we have
\be
 v_{+0}^3(z, \bar z) = \frac{1}{2} \varphi^3 (z, \bar z).
\ee
 Then the constant term of $(1, 3)$- and $(2, 3)$-entries
 in \eqref{eq:antialphamatrix} can be rephrased as 
\be
 \frac{1}{2} e^{- \frac{1}{2}\mu_k \varphi^3} \varphi^k_{\bar z} -  (v^k_{+0})_{\bar z} =0, \;\;(k =1, 2).
\ee
 This equation is the inhomogeneous Cauchy-Riemann equation and the solution can 
 be explicitly given by \cite[Theorem 1.2.2]{Hormander}:
\be
 v^k_{+0}(z, \bar z) = 
 \frac{1}{4 \pi i}\int_{\mathbb D} \frac{ e^{- \frac{1}{2} \mu_k \varphi^3(\xi, \bar \xi)}\partial_{\bar \xi} \varphi^k(\xi, \bar \xi)}{\xi-z} \> d\xi \wedge d \bar \xi,\;\;\;(k=1, 2).
\ee
 Finally a straightforward computation shows that the normalized potential 
 $\xi$ is given by \eqref{eq:xi}.
\end{proof}
 We now give the converse procedure.

 {\bf Step 1:} Let us choose a pair of normalized potentials 
 $\N_{-} (z, \l) = (\l^{-1} \xi(z), -\l^{-1}\xi(z))$ with $\xi$ in \eqref{eq:xi} and 
 consider the following linear ordinary differential equation: 
\be
 d C = C (\l^{-1}\xi),\;\;\mbox{with}\;\;C(z_*) = \id,
\ee
 where $z_*$ is some base point in some simply-connected domain $\mathbb D$ in 
 $\mathbb C$.
 It is easy to see that the solution can be computed explicitly as 
\be
 C(z, \l) =
 \begin{pmatrix}
  e^{\l^{-1} \mu_1 \Xi(z)} & 0 & \l^{-1}\int_{z*}^z \xi^1 (t)e^{ \l^{-1}\mu_1\Xi(t)}\>dt\\
  0 & e^{\l^{-1} \mu_2 \Xi(z)} &  \l^{-1}\int_{z*}^z \xi^2 (t)e^{\l^{-1}\mu_2\Xi(t)}\>dt\\
  0 & 0 & 1
 \end{pmatrix},
\ee
 where $\Xi (z) =\int_{z_*}^z \xi^3 (t)\> dt$. Let $\R_{-}$ denote the pair 
\be
 \R_{-}(z, \l) = (C (z, \l), \;\;C(z, -\l)) \in \NLGCD.
\ee

 {\bf Step 2:} 
 The Iwasawa decomposition in Theorem \ref{thm:Iwasawa} for $\R_{-}$ gives 
\be
 \R_{-} = \F \W_+, \;\;(\F \in \LGD_{\sigma}, \; \W_+ \in \PLGCD_{\sigma}),
\ee
 where $\F(z, \bar z, \l) =(F(z, \bar z, \l), F(z, \bar z, -\l))$ 
 and $\W_+ (z, \bar z, \l) =(W_+(z, \bar z, \l), W_+(z, \bar z, -\l))$.
 Here we define the {\it real part} of a loop $f(z, \bar z, \l)$ by
\be
 \Re (f(z, \bar z, \l)) = 
 \frac{1}{2}\left(f(z, \bar z, \l) + \overline{f(z, \bar z, 1/\bar \l)}\right),
\ee
 which is of course real for $\l \in \mathbb S^1$.
 Then using Theorem \ref{thm:Iwasawa-sol} the first component $F(z,\bar{z},\lambda)$ of 
 $\mathcal{F}(z,\bar{z},\lambda)$ 
 is given explicitly as follows:
\be
 F(z,\bar{z},\lambda)=
 \left(
 \begin{array}{cccc}
  e^{2 \mu_1 \Re (\l^{-1}\Xi(z))} & 0 &  e^{2 \mu_1  \Re (\l^{-1}\Xi(z)) } 
 \tilde f^{1}(z,\bar{z},\lambda) \\
  0 & e^{2 \mu_2 \Re (\l^{-1}\Xi(z))} &  e^{2 \mu_2 \Re (\l^{-1}\Xi(z)) } 
 \tilde f^{2}(z,\bar{z},\lambda)\\
  0 & 0 & 1
 \end{array}
 \right),
\ee
 where $\tilde f^{k}(z,\bar{z},\lambda), \;(k =1, 2),$ are real parts of the Iwasawa 
 decomposition of the scalar loop, that is, 
\be
g^k = \tilde f^k + g^k_+ \;\;\mbox{for}\;\; g^k(z, \bar z, \l) = \l^{-1} e^{-2\mu_k  \Re (\l^{-1}\Xi(z)) } 
 \int_{z_*}^z \xi^{k}(t) e^{\l^{-1} \mu_k  \Xi(t)}\>dt, 
\ee
 where $g^k_+$ denotes the positive element.
 Using the expansion $g^k(z, \bar z, \l) = \sum_{j =-\infty}^{\infty}
 g_j^k(z, \bar z) \l^j$, $\tilde f^k$ and $g^k_+$ are given by
\be
\left\{
\begin{array}{l}
 \tilde f^{k}(z, \bar z, \l) = \sum_{j< 0} \left(g^k_j (z, \bar z)\l^j  
 +  \overline{g^k_j(z, \bar z)} \; \l^{-j}\right),  \\[0.2cm]
 g^k_+(z, \bar z, \l) =\sum_{j \geq 0}  g^k_j(z, \bar z) \l^j
 - \sum_{j< 0}\overline{g^k_j (z, \bar z)} \; \l^{-j}.  
\end{array}
\right.
\ee

 {\bf Step 3:} Then the extended solution $\hat{F}(z,\bar{z},\lambda)=
 F(z,\bar{z},\lambda)F(z,\bar{z},1)^{-1} =
 (\hat f^1, \hat f^2, \hat f^3)$ is given by 
\be
\left\{
\begin{array}{ll}
\hat{f}^k(z,\bar{z},\l)
 &= e^{2 \mu_k \Re (\l^{-1} \Xi(z))}(\tilde f^k(z, \bar z, \l) - \tilde f^k(z, \bar z, 1)), \;(k =1, 2),  \\
 \hat f^3(z,\bar{z},\lambda) & =2 \Re ((\lambda^{-1}-1) \Xi(z)).
\end{array}
\right.
\ee
 We summarize the discussion above in the following theorem.
\begin{Theorem}\label{thm:repforsol}
  Any ${}^{(0)}\nabla$-harmonic map 
 $\varphi=(\varphi^1,\varphi^2,\varphi^3): \mathbb{D}\to G(\mu_1,\mu_2)$ 
 can be represented by some meromorphic functions $\xi^k(z), \;(k =1, 2, 3),$ 
 as follows$:$
\begin{align*}
 \varphi^{1}(z,\bar{z})&
 = \exp \left(-2 \mu_1 \Re \int_{z_*}^z \xi^3(t) \> dt \right)
 \left( \tilde f^{1}(z, \bar z, -1) -\tilde f^{1}(z, \bar z, 1)\right), \\ 
 \varphi^{2}(z,\bar{z})&= 
 \exp \left(-2 \mu_2 \Re \int_{z_*}^z \xi^3(t) \> dt \right)
 \left( \tilde f^{2}(z, \bar z, -1) -\tilde f^{2}(z, \bar z, 1)\right), \\
\varphi^{3}(z,\bar{z})&= -4 \mathrm{Re} \int_{z_*}^z \xi^3(t) \>dt.
\end{align*}
 Here $\tilde f^k(z, \bar z, \l), \;(k =1, 2),$ is the real part of the Iwasawa 
 decomposition of the scalar loop $g^k = \tilde f^k + g^k_+$ with
\be
  g^k(z, \bar z, \l) = \l^{-1}\exp \left(-2\mu_k  
 \Re (\l^{-1}\int_{z_*}^z \xi^3(t) \> dt) \right)
 \int_{z_*}^z \xi^{k}(t) \exp\left(\l^{-1} \mu_k \int_{z_*}^t \xi^3(s) \> ds\right)\>dt.
\ee
\end{Theorem}
\begin{Remark}
\mbox{}
\begin{enumerate}
 \item An integral representation formula for conformal maps into $G(\mu_1,\mu_2)$ which are 
 harmonic with respect to the metric was obtained in \cite{InoguchiLee}.

\item  In case $\mu_1=\mu_2=0$, this formula reduces to
 the classical formula of harmonic functions:
\begin{equation*}
\varphi^{k}(z,\bar{z})= -4\>\mathrm{Re}\int_{z_*}^{z} \xi^k(t)\>dt, \;\;(k =1, 2, 3).
\end{equation*}
\end{enumerate}
\end{Remark}
 Let us define $3 \times 3$ matrices $E_{p q}$ 
 with $(i, j)$-entry one if $(i, j) =(p, q)$ and zero otherwise.
 Let us consider a map $\varphi(x,y)=\exp(xE_{11})\cdot \exp(yE_{22})$.
 As we have seen before, $\varphi$ is $\neutral$-harmonic. 
 With respect to the left-invariant Riemannian metric
 $ds^2 = e^{-2\mu_1 x^3 }(dx^1)^2 + e^{-2\mu_2 x3} (dx^2)^2 + (dx^3)^2$, 
 $\varphi$ is a flat surface with constant mean 
 curvature $(\mu_1+\mu_2)/2$. 
 Moreover, one can check that this flat surface is 
 the plane $x^3=\mathrm{constant}$.
\begin{enumerate}
\item If $(\mu_1,\mu_2)=(0,0)$, then
$\varphi$ is a totally geodesic plane in the Euclidean $3$-space $\mathbb{E}^3$.  
\item If $\mu_1 =0, \mu_2 =c \neq 0$, then $\varphi$ is a cylinder over a horizontal
 horocycle in $\mathbb H^2 (-c^2)$.
\item If $\mu_1=\mu_2=c\not=0$, then $\varphi$ is a horosphere in 
the hyperbolic $3$-space $\mathbb{H}^3(-c^2)$.
\item If $\mu_1=-\mu_2\not=0$, then $\varphi$ is 
a non-totally geodesic minimal surface.
\end{enumerate}
 Thus horospheres in $\mathbb{H}^3(-c^2)$ are 
 non-harmonic with respect to the metric, but harmonic with 
 respect to the neutral connection.
 We can reconstruct these surfaces from normalized potentials with $\xi^3(z)=0$.
 It is easy to see that $C(z, \l)$ and the unitary part of 
 its Iwasawa decomposition $F$ given by $C = FV_+$ with $F \in \Lambda G(\mu_1,\mu_2)$ 
 and  $V_+ \in \Lambda^+ G(\mu_1,\mu_2)^{\C}$ are as follows:
\be
 C(z, \l) = 
\begin{pmatrix} 
  1 & 0 & \l^{-1} \int_{z*}^z  \xi^1(t) \>dt \\
  0 & 1 &  \l^{-1} \int_{z*}^z \xi^2(t)  \>dt \\
  0 & 0 & 1
\end{pmatrix}, \;
 F(z,\bar{z}, \l) =
 \begin{pmatrix}
  1 & 0 & 2 \mathrm{Re}(\l^{-1} \int_{z*}^z  \xi^1(t) \>dt)\\[0.1cm]
  0 & 1 & 2 \mathrm{Re}(\l^{-1} \int_{z*}^z  \xi^2(t) \>dt) \\
  0 & 0 & 1
 \end{pmatrix}.
\ee
 Hence we obtain the extended solution
\be
\hat{F}(z,\bar{z},\lambda)=
F(z,\bar{z},\lambda)
F(z,\bar{z},1)^{-1}=
\begin{pmatrix}
  1 & 0 & 2 \Re\left((\l^{-1} -1)\int_{z*}^z  \xi^1(t) \>dt\right)\\[0.1cm]
  0 & 1 & 2 \Re\left((\l^{-1} -1)\int_{z*}^z  \xi^2(t) \>dt\right)\\
  0 & 0 & 1
 \end{pmatrix}.
\ee
 For instance, choosing $z_*=0, \xi^1 (z)=-1/4$ and  $\xi^2 (z) =i/4$  we obtain 
 the plane $\varphi(z,\bar{z}, \l=-1)=(x,y,0)$.
\begin{Remark}
 The generalized Weierstrass type representation of 
 $\neutral$-harmonic maps into the solvable Euclidean motion group $\mathrm{SE}_2$ 
 can be given along the same line.
 Conjugated  $\mathrm{SE}_2$ by the element
\be
\frac{1}{\sqrt{2}}
 \begin{pmatrix} 
 1 &  \sqrt{-1} & 0 \\ 
 \sqrt{-1} &  1& 0 \\ 
 0 & 0 & \sqrt{2}
 \end{pmatrix},
\ee
 we obtain the isomorphic solvable group
\be
 \mathrm{SE}_2
 =
 \left\{
 (x^1, x^2, x^3) = 
 \begin{pmatrix}
 e^{\sqrt{-1} x^3} & 0 & \frac{1}{\sqrt{2}} (x^1 + \sqrt{-1} x^2) \\
 0 & e^{- \sqrt{-1} x^3} & \frac{\sqrt{-1}}{\sqrt{2}} (x^1 - \sqrt{-1} x^2) \\
 0 &0 & 1
 \end{pmatrix}
\right\} = {\rm U}_1 \ltimes \mathbb C.
\ee
 This means that $\mathrm{SE}_2$ is a real form of 
 the complexified solvable Lie group $G(1, -1)^{\mathbb C}$.
 Moreover, the $\neutral$-harmonicity of a map 
 $\varphi=(\varphi^1, \varphi^2, \varphi^3)$ can be rephrased as 
\begin{equation}\label{eq:SE2-rephrased}
 \tilde \varphi^1_{z\bar z} - \frac{1}{2}(\tilde \varphi^1_{z}\tilde \varphi^3_{\bar z}+\tilde \varphi^1_{\bar z}\tilde \varphi^3_z)=0,\; \;
  \tilde \varphi^2_{z\bar z} + \frac{1}{2}(\tilde \varphi^2_{z}\tilde \varphi^3_{\bar z}+\tilde \varphi^2_{\bar z}\tilde \varphi^3_z)=0,
\end{equation}
 where 
 $\tilde \varphi^1= \frac{1}{\sqrt{2}} (\varphi^1 + \sqrt{-1}\varphi^2)$,
 $\tilde \varphi^2= \frac{1}{\sqrt{2}} (\varphi^1 - \sqrt{-1}\varphi^2)$ 
and $\tilde \varphi^3 =\sqrt{-1} \varphi^3$.
 This is the $\neutral$-harmonicity equation of the solvable Lie group 
 $G(\mu_1, \mu_2)$ with $\mu_1 = -\mu_2 =1$ in \eqref{sol-0-harm}.
 So the generalized Weierstrass type representation of 
 $\neutral$-harmonic maps into $\mathrm{SE}_2$ is quite similar to 
 the case of the solvable Lie group $G(1, -1)$. 
\end{Remark} 
\begin{Remark}
 Solvable Lie groups equipped with a
 left invariant Riemannian metric 
 play an important and fundamental role in many branches of mathematics. 
 A Riemannian manifold is said to be a \textit{solvmanifold} if it has a 
 transitive solvable Lie group of isometries. 
 It has been shown by Heintze \cite{Heintze} that 
 connected homogeneous Riemannian manifolds
 of strictly negative curvature are solvable Lie groups with left invariant metrics.
 Moreover such spaces are diffeomorphic to $N\times \mathbb{R}^{+}$,
 where $N$ is a simply connected nilpotent Lie group.
 Azencott and Wilson \cite{AW} showed that every connected, 
 simply-connected, homogeneous Riemannian manifold of non-positive curvature admits a 
 solvable Lie group acting simply-transitively by isometries. 
 Any simply connected solvmanifold of non-positive curvature is 
\textit{standard} in the sense of 
\cite{AW}. 

 Solvable Lie groups with 
 left invariant Riemannian metric provide many examples of homogeneous 
 Einstein manifolds of non-positive curvature, for example,
 Damek-Ricci Einstein spaces and 
 Boggino-Damek-Ricci Einstein spaces \cite{BTV, Heber}. 
 Lauret proved that any Einstein solvmanifold is standard \cite{Lauret}. 
 Alekseevskii \cite{Alek} conjectured that for every 
 non flat non-compact homogeneous Einstein manifold $G/K$, 
 the isotropy subgroup $K$ must be a maximal compact subgroup of 
 $G$.
 If this conjecture is true, then the classification of non-compact 
 homogeneous Einstein manifolds is reduced to the 
 investigation of solvable Lie groups with left invariant Einstein metrics. 
 This conjecture is still an open problem.

 The class of Carnot spaces contains all rank one 
 Riemannian symmetric spaces of non-compact type. 
 According to Pansu \cite{Pansu}, a
 simply-connected solvable Lie group $G$ equipped with a left 
 invariant Riemannian metric 
 $ds^2=\langle\cdot,\cdot\rangle$ is said to be a 
 $k$-\textit{step Carnot space} if $G$ is the semidirect product of a 
 nilpotent Lie group $N$ and the abelian group $\mathbb{R}^{+}$.
 Then the Lie algebras $\mathfrak{g} = {\rm Lie} (G)$ 
 and $\mathfrak{n} = {\rm Lie}(N)$ satisfy:
 $\mathfrak{g}= \mathfrak{n}\oplus\mathbb{R}\>H$,
\be
\mathfrak{n}=\sum_{i=1}^{k}\mathfrak{n}_{i}, 
\ \
\mathfrak{n}_{i}= 
\{X\in\mathfrak{n}\ | \mathrm{ad}(H)X=i X\},\
\; (i=1,\dots ,k),\ \ 
\langle H,H\rangle=1,\ \ H\perp \mathfrak{n}
\ee
and $(G,ds^2)$ is of negative curvature \cite{Pansu}. 
 For example, 
 real hyperbolic space is a 1-step Carnot 
 space and the other rank 1 symmetric spaces of non-compact type 
 are 2-step Carnot spaces. Since harmonic maps into Riemannian symmetric spaces of
 non-compact type can be investigated 
 by the generalized Weierstrass type representation 
 \cite{DPW}, 
 it seems to be interesting to study the generalized 
 Weierstrass type representation for \textit{non-symmetric} Carnot spaces.
\end{Remark}

\bibliographystyle{plain}
\def\cprime{$'$}

\end{document}